\documentclass[reqno]{amsart}
\usepackage{amssymb}
\usepackage[usenames, dvipsnames]{color}
\usepackage{pxfonts}
\usepackage{enumerate}
\usepackage{amsmath}
\usepackage[driverfallback=dvipdfm]{hyperref}{\tiny}
\usepackage{verbatim}

\allowdisplaybreaks[4]

\numberwithin{equation}{section}

\newtheorem{theorem}{Theorem}[section]

\newtheorem{lemma}[theorem]{Lemma}
\newtheorem{prop}[theorem]{Proposition}

\theoremstyle{definition}
\newtheorem{remark}[theorem]{Remark}

\theoremstyle{definition}
\newtheorem{definition}[theorem]{Definition}

\theoremstyle{definition}

\theoremstyle{definition}

\makeatletter
\def\dashint{\operatorname%
{\,\,\text{\bf-}\kern-.98em\DOTSI\intop\ilimits@\!\!}}
\makeatother

\def\\det{\text{det}}

\def\.5{\frac{1}{2}}

\def\cD{\mathcal{D}}

\def\cQ{\mathcal{Q}}

\newcommand{\RN}[1]{%
  \textup{\uppercase\expandafter{\romannumeral#1}}%
}

\newcounter{marnote}

\makeatletter
\@namedef{subjclassname@2020}{%
  \textup{2020} Mathematics Subject Classification}
\makeatother

\begin{document}

% ------------------------------------------------------------------------

\title[Higher order parabolic systems]{Higher order parabolic systems with piecewise DMO and H\"{o}lder continuous coefficients}

\author[H. Dong]{Hongjie Dong}
\address[H. Dong]{Division of Applied Mathematics, Brown University, 182 George Street, Providence, RI 02912, USA}
\email{Hongjie\_Dong@brown.edu}
\thanks{H. Dong was partially supported by the NSF under agreement DMS-2055244 and DMS-2350129.}

\author[H. Li]{Haigang Li}
\address[H. Li]{School of Mathematical Sciences, Beijing Normal University, Laboratory of MathematiCs and Complex Systems, Ministry of Education, Beijing 100875, China.}
\email{hgli@bnu.edu.cn}
\thanks{H. Li was partially supported by  the Fundamental Research Funds for the Central Universities (2233200015), and Beijing NSF (1242006).}

\author[L. Xu]{Longjuan Xu}
\address[L. Xu]{Academy for Multidisciplinary Studies, Capital Normal University, Beijing 100048, China.}
\email{longjuanxu@cnu.edu.cn}
\thanks{L. Xu was partially supported by NSF of China (12301141), and Beijing Municipal Education Commission Science and Technology Project (KM202410028001).}

\begin{abstract}
In this paper, we are concerned with divergence form, higher-order parabolic systems in a cylindrical domain with a finite number of subdomains. We establish $L_\infty$ and Schauder  estimates of solutions when the leading coefficients  and the non-homogeneous term exhibit piecewise Dini mean oscillation and piecewise H\"{o}lder continuity, respectively. To the best of our knowledge, our results are new for higher-order elliptic and parabolic systems.
\end{abstract}

\maketitle

\section{Introduction and main results}%\label{intro}
In this paper, we consider parabolic systems in divergence form of order $2m$
\begin{equation}\label{systems}
{\bf u}_{t}+(-1)^{m}\sum_{|\alpha|\leq m, |\beta|\leq m}D^{\alpha}(A^{\alpha\beta}D^{\beta}{\bf u})=\sum_{|\alpha|\leq m}D^{\alpha}{\bf f}_{\alpha}\quad\mbox{in}~\cQ,
\end{equation}
where $\cQ:=(-T,0)\times\cD$ is a cylindrical domain,  $T\in(0,\infty)$ and
$\cD$ is a bounded domain in $\mathbb R^{d}$, $d\ge 2$. Here 
$${\bf u}=(u^1,\dots,u^n)^{\top},\quad {\bf f}_{\alpha}=(f_\alpha^1,\dots,f_\alpha^n)^{\top}$$
are (column) vector-valued functions. We assume that $\cQ$ contains $M$ disjoint time-dependent subdomains $\cQ_{j},j=1,\ldots,M$, and the interfacial boundaries are $C^{1,\text{Dini}}$ in the spatial variables and $C^{\gamma_{0}}$ in the time variable, where $\gamma_{0}>\frac{1}{2m}$.  By approximation, we also assume that any point $(t,x)\in\cQ$ belongs to the boundaries of at most two of the $\cQ_{j}$'s.

Suppose that $A^{\alpha\beta}=[A^{\alpha\beta}_{ij}]_{i,j=1}^{n}$ are $n\times n$ complex-valued matrices, which are bounded and measurable. For any $(t,x)\in\mathbb R^{d+1}$ and $\xi=(\xi_{\alpha})_{|\alpha|=m}, \xi_{\alpha}\in\mathbb C^{n}$, the leading coefficients matrices $A^{\alpha\beta}$, $|\alpha|=|\beta|=m$, satisfy 
\begin{equation}\label{ellipticity}
\nu|\xi|^{2}\leq \sum_{|\alpha|=|\beta|= m}\mathfrak{R}(A^{\alpha\beta}\xi_{\beta},\xi_{\alpha})\leq \nu^{-1}|\xi|^{2},
\end{equation}
where $\nu>0$ and $\mathfrak{R}(f)$ denotes the real part of $f$. The lower order coefficients $A^{\alpha\beta}$, $|\alpha|\neq m$ or $|\beta|\neq m$, are bounded by a constant $\Lambda\geq 1$.

The aim of this paper is to investigate $L_\infty$ and Schauder estimates of solutions to \eqref{systems} when the leading coefficients possess piecewise Dini mean oscillation (DMO) and piecewise H\"{o}lder continuity, respectively. Refer to \eqref{def omega A} for the precise definition of piecewise DMO. The main feature of this type of coefficients is that they are permitted to lack regularity assumptions in one spatial direction. This means that for every small cylinder, there exists a measurable direction, say $x_d$, such that coefficients can be merely measurable in $x_d$ and exhibit small mean oscillations in the orthogonal directions. This is the first difficulty we need to overcome. For this, we fix a coordinate system according to the geometry of
the subdomains and  prove a weak-type $(1,1)$ estimate for solutions to parabolic systems with coefficients depending only on $x_d$ by using a duality argument. Then we make use of Campanato’s method in the $L_{1/2}$-setting and  perturbation arguments on $\widetilde{D}^\alpha{\bf u}$ and ${\bf U}$ to get the mean oscillation estimates of solutions as showed in Lemma \ref{lemma itera}. This method was first introduced in \cite{c1963} and further developed in \cite{g1983,s1996,k2008,d2012}. When the leading coefficients are piecewise H\"{o}lder continuous, we shall prove Schauder estimates of the solution to \eqref{systems}. This type of coefficients has been studied by several  works for second-order equations; see, for instance, \cite{lv2000,ln2003,dx2019,dx2021}. 

The investigation of parabolic equations and systems with discontinuous coefficients is closely tied to the mechanics of membranes and the theory of stochastic processes. See \cite{ckv,k2004} and references therein. Specifically, the equation \eqref{systems} with $m=1$ is partially motivated by the study of interface problems which appear in many
practical applications, such as in composite materials with closely spaced interfacial boundaries. This scenario is  described in terms of linear second-order divergence type systems with coefficients that can have jump discontinuities in one direction. The primary goal in this context is to study the estimate of the gradient of solutions, representing the stress  from an engineering point of view. Relevant references for elliptic equations include \cite{ba1999,lv2000,ln2003,xb2013,dl2019,dx2019}, while references for parabolic equations include \cite{fknn,ll2017,dx2021}, and for Stokes and Navier-Stokes equations, refer to \cite{cdx2022}. As an extension, it is natural to consider the equation \eqref{systems} with $m\geq 2$. Compared to second-order equations, the proof in the higher-order case is more intricate. See the argument at the end of Step 2 in the proof of Proposition \ref{main prop} for further details. We would like to remark that the local H\"{o}lder regularity is improved from $\mathcal{C}^{1/2}$ in \cite[Lemma 4.1]{dk2011} to $\mathcal{C}^{1}$ in Lemma \ref{lemma xn}, which is necessary for the proof of our main results.

Now let us review  some related results for higher-order systems with discontinuous coefficients in the literature. There have been many results on the $L_p$ $(1<p<\infty)$ theory of elliptic and parabolic equations. See \cite{cff,aq2000,hhh,ps2005,ps2006,mms} for higher-order equations with vanishing mean oscillation coefficients (VMO), and \cite{d1995} for elliptic equations
with measurable coefficients for a restricted range of $m$ or $p$, and \cite{b2009,bw2008,hlw} for  results of fourth order elliptic and parabolic equations. The case when the coefficients have locally small mean oscillations with respect to the spatial variables (and measurable in the time variable in the parabolic case) was studied in \cite{dk2011ARMA}, where both divergence and non-divergence
form parabolic and elliptic systems are considered. In their subsequent work \cite{dk2011}, they investigated the case when  leading coefficients are assumed  to be merely measurable in one spacial direction and have small mean oscillations in the orthogonal directions on each small cylinder. Such coefficients have weaker regularity assumptions compared to those found in the existing literature. It is worth mentioning that the results in \cite{dk2011ARMA,dk2011} focus on the homogeneous Dirichlet boundary conditions. The conormal problem was addressed in  \cite{dk2012} and \cite{dz2016}. 

For the higher-order systems, the classical Schauder estimates were established under the assumption that the coefficients are smooth in both space and time. We refer to \cite{f1983,l1995} and  the references therein. Boccia considered higher-order non-divergence type parabolic systems in the whole space \cite{b2013}, where the coefficients are assumed to be only measurable in the time variable and H\"{o}lder continuous in the space variables.  With the same class of coefficients as in \cite{b2013}, Dong and Zhang \cite{dz2015} obtained Schauder estimates for both divergence and non-divergence type higher-order parabolic systems on a half space with the Dirichlet boundary conditions. They further studied the conormal problem in \cite{dz2016}.

In the present paper, we consider a scenario where the coefficients in the parabolic systems may have discontinuity across interfacial boundaries. This assumption distinguishes our study from the assumptions made in \cite{b2013,dz2015,dz2016}, where the coefficients were assumed to be measurable in time but possess small mean oscillations in all spatial directions. Additionally, we extend our analysis to include non-cylindrical subdomains, which constitutes a more general case in the study of parabolic systems. By considering these more general settings, this paper contributes to a broader understanding of the regularity properties of higher-order parabolic systems with discontinuous coefficients.

%\subsection{Main results}\label{mainresult}
To state our main results, we first introduce some notation. By $\mathbb R^{d}$ we mean a $d$-dimensional Euclidean space, a point in $\mathbb R^{d}$ is denoted by $x=(x_1,
\dots,x_d)=(x',x_d)$. Let $\mathbb R^{d+1}=\{(t,x): t\in \mathbb R, x\in \mathbb R^{d}\}$, 
and 
\begin{align*}
B_r(x)= \{ y\in \mathbb R^{d} : |y-x|<r \},\quad
Q_{r} (t,x)= (t-r^{2m},t)\times B_r(x).  
\end{align*} 
Denote $B_r:=B_r(0)$ and $Q_{r} :=Q_{r} (0)$. The parabolic distance between two points $z_1=(t,x)$ and $z_2=(s,y)$ is defined by
\begin{equation*}%\label{para dis}
|z_1-z_2|_{p}:=|t-s|^{\frac{1}{2m}}+|x-y|.
\end{equation*}

We  assume that $A^{\alpha\beta}$ are of piecewise Dini mean oscillation  in $\cQ$, $|\alpha|=m$, $|\beta|\leq m$, that is, 
\begin{align}\label{def omega A}
\omega_{A^{\alpha\beta}}(r):=\sup_{z_{0}\in \cQ}\inf_{\hat{A}^{\alpha\beta}\in\mathcal{A}}\fint_{Q_{r}(z_{0})}|A^{\alpha\beta}(z)-\hat{A}^{\alpha\beta}|\ dz
\end{align}
satisfies the Dini condition (see, Definition \ref{piece Dini}), where $Q_{r}(z_{0})\subset \cQ$, and $\mathcal{A}$ is the set of piecewise constant functions in each $\cQ_{j}$. 

\begin{definition}\label{piece Dini}
We say that a continuous increasing function $\omega: [0,1]\rightarrow\mathbb R$ satisfies the Dini condition provided that $\omega(0)=0$ and
$$\int_{0}^{r}\frac{\omega(s)}{s}\ ds<+\infty,\quad\forall~r\in(0,1).$$
\end{definition}

Let us define the solution spaces $\mathcal{H}_{p}^{m}(\cQ)$ as follows. Set
$$\mathbb{H}_{p}^{-m}(\cQ):=\bigg\{{\bf f}: {\bf f}=\sum_{|\alpha|\leq m}D^{\alpha}{\bf f}_{\alpha}, {\bf f}_{\alpha}\in L_{p}(\cQ)\bigg\},$$
$$\|{\bf f}\|_{\mathbb{H}_{p}^{-m}(\cQ)}:=\inf\bigg\{\sum_{|\alpha|\leq m}\|{\bf f}_{\alpha}\|_{L_{p}(\cQ)}: {\bf f}=\sum_{|\alpha|\leq m}D^{\alpha}{\bf f}_{\alpha}\bigg\},$$
and
$$\mathcal{H}_{p}^{m}(\cQ):=\{{\bf u}: {\bf u}_{t}\in \mathbb{H}_{p}^{-m}(\cQ), D^{\alpha}{\bf u}\in L_{p}(\cQ), 0\leq|\alpha|\leq m\},$$
$$\|{\bf u}\|_{\mathcal{H}_{p}^{m}(\cQ)}:=\|{\bf u}_{t}\|_{\mathbb{H}_{p}^{-m}(\cQ)}+\sum_{|\alpha|\leq m}\|D^{\alpha}{\bf u}\|_{L_{p}(\cQ)}.$$
For $\cQ=(-T,0)\times\cD$, the space $\mathcal{\mathring{H}}_{p}^{m}(\cQ)$ is the closure of $C_0^\infty([-T,0]\times\cD)$ in $\mathcal{H}_{p}^{m}(\cQ)$. Set $\partial_{p} \cQ:=\big((-T,0) \times \partial\cD\big) \cup \big(\{-T\}\times \overline \cD\big)$. 
%{\color{blue}Denote the $C^{1/2,m}$ semi-norm by
%$$[{\bf u}]_{1/2,m;\cQ}:=\sup_{\begin{subarray}{1}(t,x),(s,y)\in\cQ\\(t,x)\neq (s,y)
%\end{subarray}}\frac{|{\bf u}(t,x)-{\bf u}(s,y)|}{|t-s|^{1/2}+|x-y|^{m}},$$
%and the $C^{1/2,m}$ norm by
%$$|{\bf u}|_{1/2,m;\cQ}:=[{\bf u}]_{1/2,m;\cQ}+\sum_{k=0}^{m}|D^k{\bf u}|_{0;\cQ},\quad \text{where}\,\,|D^k{\bf u}|_{0;\cQ}=\sup_{\cQ}|D^k{\bf u}|.$$}

For $\varepsilon>0$ small, we set
$$\cD_{\varepsilon}:=\{x\in \cD: \mbox{dist}(x,\partial \cD)>\varepsilon\}.$$
Let $\widetilde{D}^\alpha{\bf u}$ be the collection of $D^\alpha{\bf u}$, $|\alpha|=m$ and $\alpha\neq me_d$. Denote
\begin{equation}\label{defU}
{\bf U}:=\sum_{|\beta|\leq m}A^{\bar{\alpha}\beta}D^{\beta}{\bf u}-(-1)^m{\bf f}_{\bar{\alpha}}
\end{equation}
with $\bar{\alpha}=me_d$. Set
\begin{align*}
\mathcal{E}_1=\|(\widetilde{D}^\alpha{\bf u},{\bf U})\|_{L_{1}(\cQ)}+\sum_{|\alpha|\leq m}\|{\bf f}_\alpha\|_{L_{\infty}(\cQ)}+\|{\bf u}\|_{L_{p}(\cQ)}+\int_{0}^{1}\frac{\tilde{\omega}_{{\bf f}_\alpha}(s)}{s}\ ds,
\end{align*}
and for $\gamma\in(0,1)$, denote
 $$\langle {\bf u}\rangle_{\gamma;\cQ}:=\sup_{\begin{subarray}{1}(t,x),(s,x)\in \cQ\\
		\quad t\neq s
\end{subarray}}\frac{|{\bf u}(t,x)-{\bf u}(s,x)|}{|t-s|^{\gamma}}.$$

The first result in this paper reads as follows.

\begin{theorem}\label{mainthm}
Let $\cQ=(-T,0)\times\cD$, $\varepsilon\in (0,1)$, $p\in(1,\infty)$, and $\gamma\in(0,1)$. Assume that $A^{\alpha\beta}$ and ${\bf f}_{\alpha}$ with $|\alpha|=m, |\beta|\leq m$ are of piecewise Dini mean oscillation in $\cQ$, and $A^{\alpha\beta}$,  ${\bf f}_{\alpha}\in L_\infty(\cQ)$ with $|\alpha|, |\beta|\leq m$, the subdomain $\cQ_j$ is $C^{1,\text{Dini}}$ in $x$ and $C^{\gamma_0}$ in $t$ with $\gamma_{0}>\frac{1}{2m}$. Let ${\bf u}\in \mathcal{H}_{p}^{m}(\cQ)$ be a weak solution to \eqref{systems} in $\cQ$. Then  ${\bf u}\in C^{1/2,m}(\overline{{\cQ}_{j}}\cap ((-T+\varepsilon,0)\times \cD_{\varepsilon}))$, $j=1,\ldots,M$, and 
\begin{equation*}
\|D^m{\bf u}\|_{L_{\infty}((-T+\varepsilon,0)\times \cD_{\varepsilon})}+\langle {\bf u}\rangle_{\frac{1}{2};(-T+\varepsilon,0)\times \cD_{\varepsilon}}\leq N\mathcal{E}_1.
\end{equation*}
Furthermore, for any fixed $z_0\in(-T+\varepsilon,0)\times\cD_\varepsilon$, there exists a coordinate system associated with $z_0$, such that for all $z\in(-T+\varepsilon,0)\times\cD_\varepsilon$, we have 
\begin{align*}
&|(\widetilde{D}^\alpha{\bf u}(z_0),{\bf U}(z_0))-(\widetilde{D}^\alpha{\bf u}(z),{\bf U}(z))|\\
&\leq N\mathcal{E}_1|z_{0}-z|_{p}^{\gamma}+N\int_{0}^{|z_{0}-z|_{p}}\frac{\tilde{\omega}_{{\bf f}_\alpha}(s)}{s}\ ds+N\mathcal{E}_1\int_{0}^{|z_{0}-z|_{p}}\frac{\tilde{\omega}_{A^{\alpha\beta}}(s)}{s}\ ds,
\end{align*}
where $N$ depends on $n,d,m,M,p,\Lambda,\nu,\varepsilon,\gamma,T$, the $C^{1,\text{Dini}}$ and $C^{\gamma_{0}}$ characteristics of $\cQ_{j}$ with respect to $x$ and $t$, respectively, and $\tilde\omega_{\bullet}(t)$ is a Dini function derived from $\omega_{\bullet}(t)$. 
\end{theorem}

Set
\begin{align*}
\mathcal{E}_2=\|(\widetilde{D}^\alpha{\bf u},{\bf U})\|_{L_{1}(\cQ)}+\sum_{j=1}^{M}\sum_{|\alpha|=m}\|{\bf f}_\alpha\|_{\mathcal{C}^{\delta}(\overline\cQ_{j})}+\sum_{|\alpha|\leq m-1}\|{\bf f}_\alpha\|_{L_{\infty}(\cQ)}+\|{\bf u}\|_{L_{p}(\cQ)}.
\end{align*}
We refer to \eqref{def-holder} for the definition of the $\mathcal{C}^{\delta}$-norm.

The next result shows the Schauder estimates when the coefficients and data exhibit piecewise H\"{o}lder continuity.
 
\begin{theorem}\label{mainthm2}
Let $\cQ=(-T,0)\times\cD$, $\varepsilon\in (0,1)$, and $p\in(1,\infty)$. Assume that $A^{\alpha\beta}$ and ${\bf f}_{\alpha}$ with $|\alpha|=m, |\beta|\leq m$ are piecewise $C^{\delta/2,\delta}$, and $A^{\alpha\beta}$,  ${\bf f}_{\alpha}\in L_\infty(\cQ)$ with $|\alpha|, |\beta|\leq m$, the subdomain $\cQ_j$ is $C^{1,\mu}$ in $x$ and $C^{\gamma_0}$ in $t$ with $\gamma_{0}>\frac{1}{2m}$. Then for any fixed $z_0\in(-T+\varepsilon,0)\times\cD_\varepsilon$, there exists a coordinate system associated with $z_0$, such that for all $z\in(-T+\varepsilon,0)\times\cD_\varepsilon$, we have 
\begin{align}\label{Dmuholder}
|(\widetilde{D}^\alpha{\bf u}(z_0),{\bf U}(z_0))-(\widetilde{D}^\alpha{\bf u}(z),{\bf U}(z))|\leq N\mathcal{E}_2|z_{0}-z|_{p}^{\delta'},
\end{align}
and 
\begin{align}\label{edt-ut}
\|D^m{\bf u}\|_{L_{\infty}((-T+\varepsilon,0)\times \cD_{\varepsilon})}+\langle{\bf u}\rangle_{\frac{m+\delta'}{2m};(-T+\varepsilon,0)\times \cD_{\varepsilon}}\leq N\mathcal{E}_2,
\end{align}
where $\delta'=\min\{\delta,\frac{\mu}{1+\mu},2m\gamma_0-1\}$, $N$ depends on $n,d,m,M,p,\Lambda,\nu,\varepsilon,T$, the $C^{1,\mu}$ and $C^{\gamma_{0}}$ characteristics of $\cQ_{j}$ with respect to $x$ and $t$, respectively.
\end{theorem}

\begin{remark}
One can apply the results in Theorem \ref{mainthm2} to the higher-order elliptic systems 
\begin{equation*}
(-1)^{m}\sum_{|\alpha|\leq m, |\beta|\leq m}D^{\alpha}(A^{\alpha\beta}D^{\beta}{\bf u})=\sum_{|\alpha|\leq m}D^{\alpha}{\bf f}_{\alpha}\quad\mbox{in}~\cD,
\end{equation*}
where $\cD$ contains $M$ disjoint subdomains $\cD_{j},j=1,\ldots,M$, the interfacial boundaries are $C^{1,\mu}$, and $A^{\alpha\beta}$ and ${\bf f}_{\alpha}$ with $|\alpha|=m, |\beta|\leq m$ are piecewise $C^{\delta}$. Indeed, we have  
\begin{align*}
\|D^m{\bf u}\|_{L_{\infty}(\cD_{\varepsilon})}\leq N\Bigg(\|(\widetilde{D}^\alpha{\bf u},{\bf U})\|_{L_{1}(\cD)}+\sum_{j=1}^{M}\sum_{|\alpha|=m}\|{\bf f}_\alpha\|_{\mathcal{C}^{\delta}(\overline\cD_{j})}+\sum_{|\alpha|\leq m-1}\|{\bf f}_\alpha\|_{L_{\infty}(\cD)}+\|{\bf u}\|_{L_{p}(\cD)}\Bigg),
\end{align*}
and for any fixed $x_0\in\cD_\varepsilon$, there exists a coordinate system associated with $x_0$, such that for all $x\in\cD_\varepsilon$, the estimate \eqref{Dmuholder} becomes 
\begin{align*}
&|(\widetilde{D}^\alpha{\bf u}(x_0),{\bf U}(x_0))-(\widetilde{D}^\alpha{\bf u}(x),{\bf U}(x))|\\
&\leq N|x_{0}-x|^{\delta_\mu}\Bigg(\|(\widetilde{D}^\alpha{\bf u},{\bf U})\|_{L_{1}(\cD)}+\sum_{j=1}^{M}\sum_{|\alpha|=m}\|{\bf f}_\alpha\|_{\mathcal{C}^{\delta}(\overline\cD_{j})}+\sum_{|\alpha|\leq m-1}\|{\bf f}_\alpha\|_{L_{\infty}(\cD)}+\|{\bf u}\|_{L_{p}(\cD)}\Bigg),
\end{align*}
where $\delta_\mu=\min\{\delta,\frac{\mu}{1+\mu}\}$, $N$ depends on $n,d,m,M,p,\Lambda,\nu,\varepsilon$, the $C^{1,\mu}$ characteristics of $\cD_{j}$.
\end{remark}

\section{Auxiliary estimates}\label{subsecauxi}
Consider the operator with coefficients depending only on $x_d$ as follows:
$$\mathcal{L}_{0}{\bf u}:=\sum_{|\alpha|=|\beta|=m}D^{\alpha}(\bar{A}^{\alpha\beta}(x_d)D^{\beta}{\bf u}).$$
Denote 
$$\bar{{\bf U}}:=\sum_{|\beta|=m}\bar{A}^{\bar{\alpha}\beta}(x_d)D^{\beta}{\bf u},$$
where $\bar{\alpha}=me_d$. Let $\widetilde{D}^\alpha{\bf u}$ be the collection of $D^\alpha{\bf u}$, $|\alpha|=m$ and $\alpha\neq me_d$. For $\gamma\in(0,1)$ and a function ${\bf u}$ defined in $Q_1$, denote
\begin{equation}\label{def-holder}
[{\bf u}]_{{\mathcal C}^\gamma(Q_1)}:=\sup_{\substack{(t,x),(s,y)\in Q_1\\(t,x)\neq(s,y)}}\frac{|{\bf u}(t,x)-{\bf u}(s,y)|}{|t-s|^{\frac{\gamma}{2m}}+|x-y|^{\gamma}},\quad\|{\bf u}\|_{{\mathcal C}^\gamma(Q_1)}:=\|{\bf u}\|_{L_\infty(Q_1)}+[{\bf u}]_{{\mathcal C}^\gamma(Q_1)}.
\end{equation}

\begin{lemma}\label{lemma xn}
Let $p\in(0,\infty)$. Assume that ${\bf u}\in C_{\text{loc}}^{\infty}(\mathbb R^{d+1})$ satisfies ${\bf u}_t+(-1)^{m}\mathcal{L}_{0}{\bf u}=0$ in $Q_{1}$. Then there exists a constant $N=N(n,d,m,p,\nu,\Lambda)$ such that 
\begin{align}\label{Dmu}
\|{\bf u}_t\|_{L_{\infty}(Q_{1/2})}+\|\widetilde{D}^\alpha{\bf u}\|_{{\mathcal C}^{1}(Q_{1/2})}+\|\bar{{\bf U}}\|_{{\mathcal C}^{1}(Q_{1/2})}\leq N\|D^m{\bf u}\|_{L_p(Q_{1})}.
\end{align}
\end{lemma}

\begin{proof}
Note that  $({\bf u}_t)_t+(-1)^{m}\mathcal{L}_{0}{\bf u}_t=0$ in $Q_{1}$. By Lemma \ref{lem loc lq} and \cite[Lemma 3.5]{dk2011}, we have for any $q\in(1,\infty)$,
\begin{align}\label{Hqmut}
\|{\bf u}_t\|_{\mathcal{H}_{q}^{m}(Q_{1/2})}\leq N\|{\bf u}_t\|_{L_{2}(Q_{2/3})}\leq N\|D^m{\bf u}\|_{L_2(Q_{3/4})}.
\end{align}
This in combination with the Sobolev embedding theorem yields 
\begin{equation}\label{est-ut}
\|{\bf u}_t\|_{L_{\infty}(Q_{1/2})}\leq N\|D^m{\bf u}\|_{L_2(Q_{3/4})}.
\end{equation}
By \cite[Lemmas 4.1 and 4.6]{dk2011}, we have 
\begin{equation}\label{Dualpha}
\|\widetilde{D}^\alpha{\bf u}\|_{L_{\infty}(Q_{3/4})}\leq N\|D^m{\bf u}\|_{L_2(Q_{1})}
\end{equation}
and
\begin{equation}\label{DU}
\|\bar{{\bf U}}\|_{L_{\infty}(Q_{3/4})}\leq N\|D^m{\bf u}\|_{L_2(Q_{1})}.
\end{equation}
From
\begin{align*}
\bar{{\bf U}}=\sum_{|\beta|=m}\bar{A}^{\bar{\alpha}\beta}(x_d)D^{\beta}{\bf u}=\bar{A}^{\bar{\alpha}\bar\alpha}(x_d)D_d^m{\bf u}+\sum_{|\beta|=m,~\beta\neq me_d}\bar{A}^{\bar{\alpha}\beta}(x_d)D^{\beta}{\bf u},
\end{align*}
it follows that
\begin{align*}%\label{est-Ddmu}
D_d^m{\bf u}=\big(\bar{A}^{\bar{\alpha}\bar\alpha}(x_d)\big)^{-1}\left(\bar{{\bf U}}-\sum_{|\beta|=m,~\beta\neq me_d}\bar{A}^{\bar{\alpha}\beta}(x_d)D^{\beta}{\bf u}\right).
\end{align*}
Then by using \eqref{ellipticity}, the boundedness of $A^{\alpha\beta}$, \eqref{Dualpha}, and \eqref{DU}, we have 
\begin{align*}
\|D_d^m{\bf u}\|_{L_\infty(Q_{3/4})}\leq N\left(\|\widetilde{D}^\alpha{\bf u}\|_{L_\infty(Q_{3/4})}+\|\bar{{\bf U}}\|_{L_\infty(Q_{3/4})}\right)\leq N\|D^m{\bf u}\|_{L_2(Q_{1})}.
\end{align*}
This together with \eqref{Dualpha} yields 
\begin{align}\label{DmuL2}
\|D^m{\bf u}\|_{L_\infty(Q_{3/4})}\leq N\|D^m{\bf u}\|_{L_2(Q_{1})}.
\end{align}
For $p\in(0,2)$, using the interpolation inequality, we obtain
\begin{align*}
\|D^m{\bf u}\|_{L_2(Q_{1})}\leq\|D^m{\bf u}\|_{L_p(Q_{1})}^{p/2}\|D^m{\bf u}\|_{L_\infty(Q_{1})}^{1-p/2}.
\end{align*}
Then combining with \eqref{DmuL2}, and using Young's inequality, we have 
\begin{align*}
\|D^m{\bf u}\|_{L_\infty(Q_{3/4})}\leq N\|D^m{\bf u}\|_{L_p(Q_{1})}^{p/2}\|D^m{\bf u}\|_{L_\infty(Q_{1})}^{1-p/2}\leq\frac{1}{2}\|D^m{\bf u}\|_{L_\infty(Q_{1})}+N\|D^m{\bf u}\|_{L_p(Q_{1})}.
\end{align*}
By an iteration argument (see, for instance, \cite[ Chapter V, Lemma 3.1]{g1983}), we have for all $p>0$,
\begin{equation}\label{estDmu}
\|D^m{\bf u}\|_{L_\infty(Q_{3/4})}\leq N\|D^m{\bf u}\|_{L_p(Q_{1})}.
\end{equation}
This together with \eqref{est-ut} gives 
\begin{equation}\label{est-ut22}
\|{\bf u}_t\|_{L_{\infty}(Q_{1/2})}\leq N\|D^m{\bf u}\|_{L_p(Q_{1})},\quad p>0.
\end{equation}

Replacing ${\bf u}$ with ${\bf u}_t$ in \eqref{estDmu} with a slightly smaller domain, and using \eqref{Hqmut}, we derive 
\begin{align}\label{Dmut}
\|D^m{\bf u}_t\|_{L_{\infty}(Q_{1/2})}\leq N\|D^m{\bf u}_t\|_{L_2(Q_{2/3})}\leq N\|D^m{\bf u}\|_{L_2(Q_{3/4})}\leq N\|D^m{\bf u}\|_{L_p(Q_{1})},\quad\forall~p>0.
\end{align}
Since $\bar{A}^{\bar{\alpha}\beta}(x_d)$ are independent of $x'$, then $D_{x'}{\bf u}$ satisfies
$$(D_{x'}{\bf u})_t+(-1)^{m}\mathcal{L}_{0}(D_{x'}{\bf u})=0\quad\mbox{in}~Q_1.$$
Similar to \eqref{Dmut},  using \cite[Lemma 3.5]{dk2011},  we obtain
\begin{align}\label{DmDu}
\|D^mD_{x'}{\bf u}\|_{L_\infty(Q_{1/2})}\leq N\|D^mD_{x'}{\bf u}\|_{L_2(Q_{2/3})}\leq N\|D^m{\bf u}\|_{L_2(Q_{3/4})}\leq N\|D^m{\bf u}\|_{L_p(Q_{1})},\quad\forall~p>0.
\end{align}
With \eqref{estDmu}, \eqref{Dmut} and \eqref{DmDu} in hand, we derive 
\begin{align*}
\|\widetilde{D}^\alpha{\bf u}\|_{{\mathcal C}^{1}(Q_{1/2})}\leq N\|D^m{\bf u}\|_{L_p(Q_{1})}.
\end{align*}
For the estimate of $\|\bar{{\bf U}}\|_{{\mathcal C}^{1}(Q_{1/2})}$, we first note that, by using \eqref{Dmut} and \eqref{DmDu}, we have 
\begin{align}\label{est-Ut}
\|\bar{{\bf U}}_t\|_{L_\infty(Q_{1/2})}=\Big\|\sum_{|\beta|=m}\bar{A}^{\bar{\alpha}\beta}(x_d)D^{\beta}{\bf u}_t\Big\|_{L_\infty(Q_{1/2})}\leq N\|D^m{\bf u}\|_{L_p(Q_{1})},\quad\forall~p>0,
\end{align}
and 
\begin{align}\label{est-DxU}
\|D_{x'}\bar{{\bf U}}\|_{L_\infty(Q_{1/2})}=\Big\|\sum_{|\beta|=m}\bar{A}^{\bar{\alpha}\beta}(x_d)D_{x'}D^{\beta}{\bf u}\Big\|_{L_\infty(Q_{1/2})}\leq N\|D^m{\bf u}\|_{L_p(Q_{1})},\quad\forall~p>0.
\end{align}
It follows from the equation ${\bf u}_t+(-1)^{m}\mathcal{L}_{0}{\bf u}=0$ that 
\begin{align*}
D_d^m\bar{{\bf U}}=(-1)^{m+1}{\bf u}_t-\sum_{\substack{|\alpha|=|\beta|=m\\\alpha_d<m}}D_d^{\alpha_d}(\bar{A}^{\bar{\alpha}\beta}(x_d)D_{x'}^{\alpha'}D^{\beta}{\bf u}),
\end{align*}
where $\alpha=(\alpha_1,\dots,\alpha_{d-1},\alpha_d)=(\alpha',\alpha_d)$. Then  combining \eqref{est-ut22}, \eqref{DmDu} with $D^mD_{x'}{\bf u}$ replaced by $D_{x'}^{\alpha'}D^{\beta}{\bf u}$, and \cite[Corollary 4.4]{dk2011}, we obtain 
\begin{equation}\label{est-DdU}
\|D_d\bar{{\bf U}}\|_{L_\infty(Q_{1/2})}\leq N\|D^m{\bf u}\|_{L_p(Q_{1})},\quad\forall~p>0.
\end{equation}
Thus, the estimate of $\|\bar{{\bf U}}\|_{{\mathcal C}^{1}(Q_{1/2})}$ follows from \eqref{est-Ut}--\eqref{est-DdU}.
\end{proof}

\begin{remark}\label{rmkadj}
Let ${\mathcal{L}_0^{*}}$ be the adjoint operator of ${\mathcal{L}_0}$, then a similar argument shows that the solution of $-{\bf u}_t+(-1)^{m}\mathcal{L}_{0}^{*}{\bf u}=0$ satisfies
\begin{align*}
\|{\bf u}_t\|_{L_{\infty}(Q_{1/2})}+\|\widetilde{D}^\alpha{\bf u}\|_{{\mathcal C}^{1}(Q_{1/2})}+\|\bar{{\bf U}}\|_{{\mathcal C}^{1}(Q_{1/2})}\leq N\|D^m{\bf u}\|_{L_p((-1,1)\times B_{1})}.
\end{align*}
\end{remark}

Next let us recall the definition of variably partially small BMO  coefficients in \cite{dk2011}: there exists a sufficiently small constant $\gamma_{0}=\gamma_{0}(d,n,m,p,\nu)\in (0,1/2)$ and a constant $r_{0}\in(0,1)$ such that for any $r\in(0,r_{0})$ and $z_0:=(t_{0},x_{0})\in Q_{1}$ with $B_{r}(x_{0})\subset B_{1}$, in a coordinate system depending on $(t_{0},x_{0})$ and $r$, one can find $\bar{A}^{\alpha\beta}(x_d):=\bar{A}_{r,z_0}^{\alpha\beta}(x_d)$ satisfying
\begin{equation}\label{BMO}
\fint_{Q_{r}(t_{0},x_{0})}|A^{\alpha\beta}(t,x)-\bar A^{\alpha\beta}(x_d)|\ dx dt\leq\gamma_{0}.
\end{equation}
Here, we used the subscript in $\bar A_{r,z_0}^{\alpha\beta}(x_d)$ to indicate that the function depends on $r$ and $z_0$. We shall give the $\mathcal{H}_{p}^{m}$-solvability for parabolic systems with leading coefficients which satisfy \eqref{BMO} in $Q_{1}$. For this, we choose a cut-off function $\eta\in C_{0}^{\infty}(B_{1})$ with
$$0\leq\eta\leq1,\quad \eta\equiv1~\mbox{in}~B_{3/4},\quad|\nabla\eta|\leq 8.$$
Denote 
\begin{equation}\label{def-tilde-A}
\tilde{A}^{\alpha\beta}=\eta A^{\alpha\beta}(t,x)+\nu(1-\eta)\delta_{\alpha\beta}\delta_{ij},
\end{equation}
where $\delta_{\alpha\beta}$ and $\delta_{ij}$ are the Kronecker delta symbols. 
Let $\tilde{\mathcal{L}\ }$ be the elliptic operator defined by
$$
\tilde{\mathcal{L}\ }{\bf u}:=\sum_{|\alpha|\leq m,~|\beta|\leq m}D^{\alpha}(\tilde{A}^{\alpha\beta}D^{\beta}{\bf u}).
$$
Then for sufficiently small $\gamma$, the coefficients $\tilde{A}^{\alpha\beta}(t,x)$ and the boundary $\partial B_{1}$ satisfy the Assumption 8.1 ($\gamma$) in \cite{dk2011}. By \cite[Theorem 8.2]{dk2011}, we have the following $\mathcal{H}_{p}^{m}$-solvability.

\begin{lemma}\label{solvability}
For any $p\in(1,\infty)$ and ${\bf f}_{\alpha}\in L_{p}(Q_{1})$, there exists a unique solution ${\bf u}\in \mathcal{\mathring{H}}_{p}^{m}(Q_{1})$ of
\begin{align*}%\label{approxi sol}
{\bf u}_t+(-1)^m\tilde{\mathcal{L}\ }{\bf u}=\sum_{|\alpha|\leq m}D^{\alpha}{\bf f}_{\alpha}&\quad\mbox{in}~Q_{1}
\end{align*}
with the initial data ${\bf u}(-1,\cdot)\equiv0$ in $B_{1}$. Furthermore, ${\bf u}$ satisfies
\begin{align*}%\label{est H u}
\|{\bf u}\|_{\mathcal{H}_{p}^{m}(Q_{1})}\leq N\sum_{|\alpha|\leq m}\|{\bf f}_{\alpha}\|_{L_{p}(Q_{1})},
\end{align*}
where $N$ depends on $d,n,m,p,\nu,\Lambda$, and $r_{0}$.
\end{lemma}

Denote
$$
\tilde{\mathcal{L}_0\ }{\bf u}:=\sum_{|\alpha|=|\beta|=m}D^{\alpha}(\tilde{A}^{\alpha\beta}D^{\beta}{\bf u}),
$$
where $\tilde{A}^{\alpha\beta}$ is defined in \eqref{def-tilde-A} with $A^{\alpha\beta}(t,x)$ replaced by $\bar{A}_{r,z_0}^{\alpha\beta}(x_d)$, and
$\eta\in C_{0}^{\infty}(B_{r}(x_{0}))$ satisfies
$$0\leq\eta\leq1,\quad\eta\equiv1~\mbox{in}~B_{2r/3}(x_{0}),
\quad|\nabla\eta|\leq {6}/{r}.$$

\begin{lemma}\label{weak est barv}
Let $p\in(1,\infty)$ and ${\bf v}\in \mathcal{H}_{p}^{m}(Q_{r_0}(z_{0}))$ be a weak solution to the problem 
\begin{align*}
\begin{cases}
{\bf v}_t+(-1)^{m}\tilde{\mathcal{L}_0\ }{\bf v}=\sum_{|\alpha|=m}D^{\alpha}({\bf F}_\alpha\chi_{Q_{r_0/2}(z_{0})})&\ \mbox{in}~Q_{r_0}(z_{0}),\\
{\bf v}=|D{\bf v}|=\dots=|D^{m-1}{\bf v}|=0&\ \mbox{on}~\partial_p Q_{r_0}(z_{0}),
\end{cases}
\end{align*}
where $r_0\in(0,1)$ and ${\bf F}_\alpha\in L_{p}(Q_{r_0/2}(z_{0}))$. Then for any $s>0$, we have for $|\beta|\leq m$, 
\begin{align}\label{Dmv}
|\{z\in Q_{r_0/2}(z_{0}): |D^{\beta}{\bf v}(z)|>s\}|\leq\frac{N}{s}r_0^{m-|\beta|}\sum_{|\alpha|=m}\|{\bf F}_\alpha\|_{L_{1}(Q_{r_0/2}(z_{0}))},
\end{align}
where $N=N(n,d,m,p,\nu)$.
\end{lemma}

\begin{proof}
For simplicity, set $z_{0}=0$, $r_0=1$, and $\bar{A}^{\alpha\beta}(x_d):=\bar{A}^{\alpha\beta}(0',x_d)$. Denote $\mathcal{M}:=(M^{\alpha\beta}(x_d))$, where $M^{\alpha\beta}(x_d)$ is defined as follows: 
\begin{equation}\label{DEF_M}
\begin{split}
M^{\alpha\beta}(x_d)&=\delta_{\alpha\beta}\quad\mbox{for}~\alpha,\beta\neq \bar\alpha(=me_d);\quad
M^{\beta\bar\alpha}(x_d)=A^{\bar\alpha\beta}(x_d)\quad\mbox{for}~|\beta|=m;\\
M^{\bar\alpha\beta}(x_d)&=0\quad\mbox{for}~\beta\neq\bar\alpha.
\end{split}
\end{equation}
For any ${\bf F}_\alpha\in L_{p}(Q_{1/2})$ with $|\alpha|=m$, let $\tilde{\bf F}=\mathcal{M}{\bf F}$ and solve for ${\bf v}$ with ${\bf F}$ in place of $\tilde{\bf F}$. It follows from Lemma \ref{solvability} that for each $\beta$ with $|\beta|\leq m$, $\mathcal{O}: {\bf F}_\alpha\to D^{\beta}{\bf v}$ is a bounded linear operator on $L_{p}(Q_{1/2})$. Fix $\bar{z}=(\bar{t},\bar{y})\in Q_{1/2}$, $0<r<1/4$, and let ${\bf b}_\alpha\in L_{p}(Q_{1})$ be supported in $Q_{r}(\bar{z})\cap Q_{1/2}$ with mean zero. Set
 $\tilde{\bf b}=\mathcal{M}{\bf b}$, and let ${\bf v}_{1}\in \mathcal{H}_{p}^{m}(Q_{1})$ be the unique weak solution of
\begin{align*}
\begin{cases}
\partial_t{\bf v}_1+(-1)^{m}\tilde{\mathcal{L}_0\ }{\bf v}_{1}=\sum_{|\alpha|=m}D^{\alpha}\tilde{\bf b}_\alpha&\ \mbox{in}~Q_{1},\\
{\bf v}_{1}=|D{\bf v}_1|=\dots=|D^{m-1}{\bf v}_1|=0&\ \mbox{on}~\partial_p Q_{1}.
\end{cases}
\end{align*}
The solvability follows from Lemma \ref{solvability}. From \cite[Lemma 3.8]{dx2021}, one can see that it suffices to show
\begin{align}\label{estDv1}
\int_{Q_{1/2}\setminus \mathcal {C}_{cr}(\bar{z})}|D^{\beta}{\bf v}_{1}|\ dxdt\leq N\sum_{|\alpha|= m}\int_{Q_{r}(\bar{z})\cap Q_{1/2}}|{\bf b}_\alpha|\ dxdt,
\end{align}
where $|\beta|\leq m$, $c=24$ and $\mathcal {C}_r(t,x):=(t-r^{2m},t+r^{2m})\times B_r(x)$. We will prove the cases of $|\beta|=m$ and $|\beta|=0$ as examples since the case of $1\leq|\beta|\leq m-1$ is proved in the same manner. 

(1) We aim to prove \eqref{estDv1} with $|\beta|=m$. For any $R\geq 24r$ such that $Q_{1/2}\setminus \mathcal {C}_{R}(\bar{z})\neq\emptyset$ and ${\bf h}_\alpha\in C_{0}^{\infty}((\mathcal {C}_{2R}(\bar{z})\setminus \mathcal {C}_{R}(\bar{z}))\cap Q_{1/2})$, let ${\bf v}_{2}\in \mathcal{H}_{p'}^{m}(\mathcal {C}_{1})$ be a weak solution of
\begin{align}\label{adjsol}
\begin{cases}
-\partial_t{\bf v}_2+(-1)^{m}\tilde{\mathcal{L}_0^{*}\ }{\bf v}_{2}=\sum_{|\alpha|=m}D^{\alpha}{\bf h}_\alpha&\ \mbox{in}~\mathcal {C}_{1},\\
{\bf v}_{2}=|D{\bf v}_2|=\dots=|D^{m-1}{\bf v}_2|=0&\ \mbox{on}~((-1,1)\times \partial B_{1})\cup (\{1\}\times \overline{B_1}),
\end{cases}
\end{align}
where $\mathcal {C}_1:=(-1,1)\times B_1$, ${1}/{p}+{1}/{p'}=1$ and $\tilde{\mathcal{L}_0^{*}\ }$ is the adjoint operator of $\tilde{\mathcal{L}_0\ }$.
Then by using the definition of weak solutions, the definition of $\mathcal{M}$ in \eqref{DEF_M}, and the assumption of ${\bf b}_\alpha$, we have
\begin{align}\label{def-weak}
\sum_{|\alpha|=m}(-1)^{m}\int_{Q_{1/2}}D^\alpha{\bf v}_{1}\cdot {\bf h}_\alpha&=\sum_{|\alpha|=m}(-1)^m\int_{Q_{1/2}}D^\alpha {\bf v}_{2}\cdot \tilde{\bf b}_\alpha\nonumber\\
&=\sum_{|\alpha|=m}(-1)^{m}\int_{Q_{r}(\bar{z})\cap Q_{1/2}}
\left(\widetilde{D}^\alpha{\bf v}_{2}, {\bf V}_{2}
\right)\cdot{\bf b}_\alpha\nonumber\\
&=\sum_{|\alpha|=m}(-1)^{m}\int_{Q_{r}(\bar{z})\cap Q_{1/2}}
\left(
\widetilde{D}^\alpha{\bf v}_{2}-\widetilde{D}^\alpha{\bf v}_{2}(\bar{z}), {\bf V}_{2}-{\bf V}_{2}(\bar{z})
\right)\cdot{\bf b}_\alpha,
\end{align}
where ${\bf V}_{2}=\sum_{|\beta|= m}\bar{A}_{r,z_0}^{\bar{\alpha}\beta}(x_d)D^{\beta}{\bf v}_{2}$. Recall that $\eta\equiv1$ in $B_{2/3}$, $\bar y\in B_{1/2}$, and $R\leq1$, we have $B_{R/12}(\bar{y})\subset B_{2/3}$. Then we find that ${\bf v}_{2}\in \mathcal{H}_{p'}^{m}(\mathcal {C}_{1})$ satisfies
$$\partial_t{\bf v}_2+(-1)^{m}\bar{\mathcal{L}_0^{*}\ }{\bf v}_{2}=0\quad\mbox{in}~\mathcal {C}_{R/12}(\bar{z}),$$
where $\bar{\mathcal{L}_0^{*}\ }$ is the adjoint operator of $\bar{\mathcal{L}_0\ }$ with coefficients depending on $x_d$. By using \eqref{Dmu} and Remark \ref{rmkadj} with a suitable scaling, $r\leq R/24$, and applying the $\mathcal{H}^m_{p}$-estimate in Lemma \ref{solvability} to \eqref{adjsol}, we have
\begin{align*}%\label{est Dv2 V2}
&\|\widetilde{D}^\alpha{\bf v}_{2}-\widetilde{D}^\alpha{\bf v}_{2}(\bar{z})\|_{L_{\infty}(Q_{r}(\bar{z})\cap Q_{1/2})}
+\|{\bf V}_{2}-{\bf V}_{2}(\bar{z})\|_{L_{\infty}(Q_{r}(\bar{z})\cap Q_{1/2})}\nonumber\\
&\leq N\frac{r}{R}R^{-{(d+2m)}/{p'}}\|D^m{\bf v}_{2}\|_{L_{p'}(\mathcal {C}_{R/12}(\bar{z}))}\nonumber\\
&\leq N\frac{r}{R}R^{-{(d+2m)}/{p'}}\sum_{|\alpha|=m}\|{\bf h}_\alpha\|_{L_{p'}((\mathcal {C}_{2R}(\bar{z})\setminus \mathcal {C}_{R}(\bar{z}))\cap Q_{1/2})}.
\end{align*}
This in combination with \eqref{def-weak} yields
\begin{align}\label{esti Dv h}
&\left|\sum_{|\alpha|=m}\int_{(\mathcal {C}_{2R}(\bar{z})\setminus \mathcal {C}_{R}(\bar{z}))\cap Q_{1/2}}D^\alpha{\bf v}_{1}\cdot {\bf h}_\alpha\right|\nonumber\\
&\leq \sum_{|\alpha|=m}\|{\bf b}_\alpha\|_{L_{1}(Q_{r}(\bar{z})\cap Q_{1/2})}\left|\left|\left(
\widetilde{D}^\alpha{\bf v}_{2}-\widetilde{D}^\alpha{\bf v}_{2}(\bar{z}), {\bf V}_{2}-{\bf V}_{2}(\bar{z})
\right)\right|\right|_{L_{\infty}(Q_{r}(\bar{z})\cap Q_{1/2})}\nonumber\\
&\leq N\frac{r}{R}R^{-{(d+2m)}/{p'}}\sum_{|\alpha|=m}\|{\bf b}_\alpha\|_{L_{1}(Q_{r}(\bar{z})\cap Q_{1/2})}\|{\bf h}_\alpha\|_{L_{p'}((\mathcal {C}_{2R}(\bar{z})\setminus \mathcal {C}_{R}(\bar{z}))\cap Q_{1/2})}.
\end{align}
Combining the duality and H\"{o}lder's inequality, we have
\begin{align}\label{dilation Dv}
\|D^m{\bf v}_{1}\|_{L_{1}((\mathcal {C}_{2R}(\bar{z})\setminus \mathcal {C}_{R}(\bar{z}))\cap Q_{1/2})}\leq N\frac{r}{R}\sum_{|\alpha|=m}\|{\bf b}_\alpha\|_{L_{1}(Q_{r}(\bar{z})\cap Q_{1/2})}.
\end{align}
Let $N_{0}$ be the smallest positive integer such that $Q_{1/2}\subset Q_{2^{N_{0}}cr}(\bar{z})$ with $c=24$. By taking $R=cr, 2cr,\ldots,2^{N_{0}-1}cr$ in \eqref{dilation Dv} and summarizing, we obtain
\begin{align*}
\int_{Q_{1/2}\setminus \mathcal {C}_{cr}(\bar{z})}|D^m{\bf v}_{1}|\ dx dt&\leq N\sum_{k=1}^{N_{0}}2^{-k}\sum_{|\alpha|= m}\|{\bf b}_\alpha\|_{L_{1}(Q_{r}(\bar{z})\cap Q_{1/2})}\\
&\leq N\sum_{|\alpha|=m}\int_{Q_{r}(\bar{z})\cap Q_{1/2}}|{\bf b}_\alpha|\ dxdt,
\end{align*}
where $N=N(n,d,m,p,\nu)$. This together with \cite[Lemma 3.8]{dx2021} yields \eqref{Dmv} for $|\beta|=m$.

(2) To prove \eqref{estDv1} with $\beta=0$, let ${\bf v}_{0}\in \mathcal{H}_{p'}^{m}(\mathcal {C}_{1})$ be a weak solution of
\begin{align*}%\label{adjsolv0}
\begin{cases}
-\partial_t{\bf v}_0+(-1)^{m}\tilde{\mathcal{L}_0^{*}\ }{\bf v}_{0}={\bf h}&\ \mbox{in}~\mathcal {C}_{1},\\
{\bf v}_{0}=|D{\bf v}_0|=\dots=|D^{m-1}{\bf v}_0|=0&\ \mbox{on}~((-1,1)\times \partial B_{1})\cup (\{1\}\times \overline{B_1}),
\end{cases}
\end{align*}
where ${\bf h}\in C_{0}^{\infty}((\mathcal {C}_{2R}(\bar{z})\setminus \mathcal {C}_{R}(\bar{z}))\cap Q_{1/2})$. Then \eqref{def-weak} becomes
\begin{align*}%\label{weak-sol}
\int_{Q_{1/2}}{\bf v}_{1}\cdot {\bf h}&=\sum_{|\alpha|=m}(-1)^m\int_{Q_{1/2}}D^\alpha {\bf v}_{0}\cdot \tilde{\bf b}_\alpha\nonumber\\
&=\sum_{|\alpha|=m}(-1)^{m}\int_{Q_{r}(\bar{z})\cap Q_{1/2}}
\left(
\widetilde{D}^\alpha{\bf v}_{0}-\widetilde{D}^\alpha{\bf v}_{0}(\bar{z}), {\bf V}_{0}-{\bf V}_{0}(\bar{z})
\right)\cdot{\bf b}_\alpha,
\end{align*}
where ${\bf V}_{0}=\sum_{|\beta|= m}\bar{A}_{r,z_0}^{\bar{\alpha}\beta}(x_d)D^{\beta}{\bf v}_{0}$.
Similar to \eqref{esti Dv h}, we have
\begin{align*}
&\left|\int_{(\mathcal {C}_{2R}(\bar{z})\setminus \mathcal {C}_{R}(\bar{z}))\cap Q_{1/2}}{\bf v}_{1}\cdot {\bf h}\right|\nonumber\\
&\leq N\frac{r}{R}R^{-{(d+2m)}/{p'}}\sum_{|\alpha|=m}\|{\bf b}_\alpha\|_{L_{1}(Q_{r}(\bar{z})\cap Q_{1/2})}\|{\bf h}\|_{L_{p'}((\mathcal {C}_{2R}(\bar{z})\setminus \mathcal {C}_{R}(\bar{z}))\cap Q_{1/2})}.
\end{align*}
Similar to \eqref{dilation Dv}, we obtain
\begin{align*}
\|{\bf v}_{1}\|_{L_{1}((\mathcal {C}_{2R}(\bar{z})\setminus \mathcal {C}_{R}(\bar{z}))\cap Q_{1/2})}\leq N\frac{r}{R}\sum_{|\alpha|= m}\|{\bf b}_\alpha\|_{L_{1}(Q_{r}(\bar{z})\cap Q_{1/2})}.
\end{align*}
The rest of the proof of \eqref{Dmv}  with $\beta=0$ is the same as the argument below \eqref{dilation Dv} and thus is omitted. 
Therefore, the proof of the lemma is completed.
\end{proof}

Next we prove a mean oscillation estimate.
\begin{lemma}\label{estDmw}
Suppose that ${\bf w}\in \mathcal{H}_{p}^{m}(Q_{r/2}(z_{0}))$ satisfies
$${\bf w}_t+(-1)^{m}\sum_{|\alpha|=|\beta|=m}D^{\alpha}(\bar{A}_{r,z_0}^{\alpha\beta}(x_d)D^{\beta}{\bf w})=0\quad\mbox{in}~Q_{r/2}(z_{0}).$$ 
Then for any $\kappa\in\big(0,1/4\big)$ and $\Theta\in \mathbb R^{n\times \tbinom{d+m-1}{m}}$, we have
\begin{align}\label{holder w bar}
&\left(\fint_{Q_{\kappa r}(z_{0})}|\widetilde{D}^\alpha{\bf w}-(\widetilde{D}^\alpha{\bf w})_{Q_{\kappa r}(z_{0})}|^{1/2}\ dxdt+\fint_{Q_{\kappa r}(z_{0})}|{\bf W}-({\bf W})_{Q_{\kappa r}(z_{0})}|^{1/2}\ dxdt\right)^{2}\nonumber\\
&\leq N\kappa\left(\fint_{Q_{r/2}(z_{0})}|(\widetilde{D}^\alpha{\bf w},{\bf W})-\Theta|^{1/2}\ dxdt\right)^{2},
\end{align}
where $\widetilde{D}^\alpha{\bf w}$ is the collection of $D^\alpha{\bf w}$, $|\alpha|=m$, $\alpha\neq me_d$, ${\bf W}=\sum_{|\beta|=m}\bar{A}_{r,z_0}^{\bar{\alpha}\beta}(x_d)D^{\beta}{\bf w}$, and $N=N(n,d,m,p,\nu,\Lambda)$.
\end{lemma}

\begin{proof}
For any $\kappa\in\big(0,1/4\big)$, by Lemma \ref{lemma xn} with a suitable scaling, we have
\begin{align}\label{DW kappa}
&\|\widetilde{D}^\alpha{\bf w}-(\widetilde{D}^\alpha{\bf w})_{Q_{\kappa r}(z_{0})}\|_{L_{1/2}(Q_{\kappa r}(z_{0}))}^{1/2}+\|{\bf W}-({\bf W})_{Q_{\kappa r}(z_{0})}\|_{L_{1/2}(Q_{\kappa r}(z_{0}))}^{1/2}\nonumber\\
&\leq N\kappa^{d+2m+\frac{1}{2}}\int_{Q_{r/2}(z_{0})}|D^m{\bf w}|^{1/2}\ dxdt\nonumber\\
&\leq N\kappa^{d+2m+\frac{1}{2}}\int_{Q_{r/2}(z_{0})}|(\widetilde{D}^\alpha{\bf w},{\bf W})|^{1/2}\ dxdt.
\end{align}
For $\Theta\in \mathbb R^{n\times \tbinom{d+m-1}{m}}$, define
\begin{align*}%\label{def hxd}
{\bf H}(x_d):={\bf H}_{r,z_0}(x_d)&=\int_{0}^{x_d}\int_0^{s_{m-1}}\dots\int_0^{s_{1}}\Big(\bar{A}_{r,z_0}^{\bar{\alpha}\bar\alpha}(s)\Big)^{-1}
\Big(\Theta_{\bar{\alpha}}-\\
&\qquad\sum_{|\beta|=m,~\beta\neq me_d}\bar{A}_{r,z_0}^{\bar{\alpha}\beta}(s)\Theta_{\beta}\Big)\ ds ds_1\dots ds_{m-1},
\end{align*}
where $\bar\alpha=me_d$. Then $D^{\beta}{\bf H}(x_d)=0$ for $|\beta|=m,~\beta\neq me_d$, and 
\begin{align}\label{DbetaH}
\sum_{|\beta|=m}\bar{A}_{r,z_0}^{\bar{\alpha}\beta}(x_d)D^{\beta}{\bf H}(x_d)&=\bar{A}_{r,z_0}^{\bar{\alpha}\bar\alpha}(x_d)D_d^{m}{\bf H}(x_d)+\sum_{|\beta|=m,~\beta\neq me_d}\bar{A}_{r,z_0}^{\bar{\alpha}\beta}(x_d)D^{\beta}{\bf H}(x_d)\nonumber\\
&=\bar{A}_{r,z_0}^{\bar{\alpha}\bar\alpha}(x_d)D_d^{m}{\bf H}(x_d)\nonumber\\
&=\Theta_{\bar{\alpha}}-\sum_{|\beta|=m,~\beta\neq me_d}\bar{A}_{r,z_0}^{\bar{\alpha}\beta}(x_d)\Theta_{\beta}.
\end{align}

For $\beta=(\beta_1,\dots,\beta_d)$, set
\begin{equation*}
{\bf F}=\sum_{|\beta|=m,~\beta\neq me_d}\frac{x_1^{\beta_1}\dots x_d^{\beta_d}}{\beta_1!\dots\beta_d!}\Theta_\beta
\end{equation*}
such that 
\begin{equation}\label{DbetaF}
D^{\beta}{\bf F}=\Theta_\beta,~|\beta|=m,~\beta\neq me_d,\quad D_d^m{\bf F}=0.
\end{equation}
Define
\begin{align*}%\label{definition tildew}
\tilde{{\bf w}}:={\bf w}-{\bf F}-{\bf H}(x_d).
\end{align*}
Then a direct calculation yields 
$$\widetilde{D}^\alpha\tilde{\bf w}=\widetilde{D}^\alpha{\bf w}-\Theta_\alpha,~|\alpha|=m,~\alpha\neq me_d,\quad \tilde{\bf W}:=\sum_{{|\beta|=m}}\bar{A}_{r,z_0}^{\bar{\alpha}\beta}(x_d)D^{\beta}\tilde{\bf w}={\bf W}-\Theta_{\bar{\alpha}}.$$
Using \eqref{DbetaH} and \eqref{DbetaF} , we obtain in $Q_{r/2}(z_{0})$,
\begin{align*}
&\tilde{\bf w}_t+(-1)^{m}\sum_{|\alpha|=|\beta|=m}D^{\alpha}(\bar{A}_{r,z_0}^{\alpha\beta}(x_d)D^{\beta}\tilde{\bf w})\\
&=(-1)^{m+1}\sum_{\substack{|\alpha|=|\beta|=m\\\beta\neq me_d}}D^{\alpha}(\bar{A}_{r,z_0}^{\alpha\beta}(x_d)D^{\beta}{\bf F})+(-1)^{m+1}\sum_{|\alpha|=m}D^{\alpha}(\bar{A}_{r,z_0}^{\alpha\bar\alpha}(x_d)D_d^{m}{\bf H}(x_d))\\
&=(-1)^{m+1}\sum_{|\beta|=m,~\beta\neq me_d}D^{\bar\alpha}(\bar{A}_{r,z_0}^{\bar\alpha\beta}(x_d)\Theta_{\beta})+(-1)^{m+1}D^{\bar\alpha}(\bar{A}_{r,z_0}^{\bar\alpha\bar\alpha}(x_d)D_d^{m}{\bf H}(x_d))\\
&=(-1)^{m+1}\sum_{|\beta|=m,~\beta\neq me_d}D^{\bar\alpha}(\bar{A}_{r,z_0}^{\bar\alpha\beta}(x_d)\Theta_{\beta})+(-1)^{m+1}D^{\bar\alpha}\Big(\Theta_{\bar{\alpha}}-\sum_{|\beta|=m,~\beta\neq me_d}\bar{A}_{r,z_0}^{\bar{\alpha}\beta}(x_d)\Theta_{\beta}\Big)\\
&=0.
\end{align*}
Now replacing ${\bf w}$ and ${\bf W}$ with $\tilde{\bf w}$ and $\tilde{\bf W}$ in \eqref{DW kappa}, respectively, we have
\begin{align*}
&\|\widetilde{D}^\alpha{\bf w}-(\widetilde{D}^\alpha{\bf w})_{Q_{\kappa r}(z_{0})}\|_{L_{1/2}(Q_{\kappa r}(z_{0}))}^{1/2}+\|{\bf W}-({\bf W})_{Q_{\kappa r}(z_{0})}\|_{L_{1/2}(Q_{\kappa r}(z_{0}))}^{1/2}\nonumber\\
&\leq N\kappa^{d+2m+\frac{1}{2}}\int_{Q_{r/2}(z_{0})}|(\widetilde{D}^\alpha{\bf w}-\Theta_\alpha,{\bf W}-\Theta_{\bar{\alpha}})|^{1/2}\ dxdt\\
&= N\kappa^{d+2m+\frac{1}{2}}\int_{Q_{r/2}(z_{0})}|(\widetilde{D}^\alpha{\bf w},{\bf W})-\Theta|^{1/2}\ dxdt.
\end{align*}
Thus, we derive \eqref{holder w bar} and the proof is finished.
\end{proof}

\section{Proof of Theorem \ref{mainthm}}\label{modulus}

In this section, we first show that the estimate of $[{\bf u}]_{1/2,m}$ is reduced to the estimates of $\|D^{k}{\bf u}\|_{L_{\infty}}$ and $\|{\bf f}_{\alpha}\|_{L_{\infty}}$, where $k=0,\dots,m$, and $|\alpha|\leq m$.

\begin{lemma}\label{lemmaC12}
Let ${\bf u}\in \mathcal{H}_{p}^{m}(\cQ)$ be a weak solution to \eqref{systems} in $\cQ$. Suppose that $\|D^k{\bf u}\|_{L_{\infty}(Q_{1/2})}<\infty$, $k=0,\dots,m$. Then
\begin{equation}\label{estclaim}
[{\bf u}]_{1/2,m;Q_{1/4}}\leq N\left(\sum_{k=0}^{m}\|D^k{\bf u}\|_{L_{\infty}(Q_{1/2})}
+\sum_{|\alpha|\leq m}\|{\bf f}_{\alpha}\|_{L_{\infty}(Q_{1/2})}\right).
\end{equation}
\end{lemma}

\begin{proof}
To prove \eqref{estclaim}, we rewrite \eqref{systems} as
\begin{align*}
{\bf u}_{t}+(-1)^{m}\sum_{|\alpha|=|\beta|= m}D^{\alpha}(A^{\alpha\beta}D^{\beta}{\bf u})&=\sum_{|\alpha|\leq m}D^{\alpha}{\bf f}_{\alpha}+(-1)^{m+1}\sum_{|\alpha|=m, ~|\beta|\leq m-1}D^{\alpha}(A^{\alpha\beta}D^{\beta}{\bf u})\\
&\quad+(-1)^{m+1}\sum_{1\leq|\alpha|\leq m-1,~ |\beta|\leq m}D^{\alpha}(A^{\alpha\beta}D^{\beta}{\bf u})\\
&\quad+(-1)^{m+1}\sum_{\alpha=0,~ |\beta|\leq m}A^{\alpha\beta}D^{\beta}{\bf u}\quad \mbox{in}~Q_{1}.
\end{align*}
Denote 
\begin{align*}
{\bf g}_{1,\alpha}:=\sum_{|\beta|\leq m-1}A^{\alpha\beta}D^{\beta}{\bf u}\quad\mbox{for}~|\alpha|=m,\quad {\bf g}_{2,\alpha}:=\sum_{|\beta|\leq m}A^{\alpha\beta}D^{\beta}{\bf u}\quad\mbox{for}~1\leq|\alpha|\leq m-1,
\end{align*}
and for $\alpha=0$,
$${\bf g}:=\sum_{|\beta|\leq m}A^{\alpha\beta}D^{\beta}{\bf u}.$$
Now let us fix $z_0\in\overline{Q}_{1/4}$ and take $r\in(0,1/4)$. By using a similar argument in the proof of \cite[Lemma 3.3]{dk2011} (see also, \cite[Lemma 3.1]{k2007}), we obtain, for any $s=0,\dots,m-1$,
\begin{align*}
\int_{Q_r(z_0)}|D^s{\bf u}-D^s{\bf P}|^2\ dz&\leq Nr^{2(m-s)}\Bigg(\sum_{|\alpha|\leq m}r^{2(m-|\alpha|)}\int_{Q_r(z_0)}|{\bf f}_{\alpha}|^2\ dz+\sum_{|\alpha|= m}\int_{Q_r(z_0)}|{\bf g}_{1,\alpha}|^2\ dz\\
&\quad+\sum_{1\leq|\alpha|\leq m-1}r^{2(m-|\alpha|)}\int_{Q_r(z_0)}|{\bf g}_{2,\alpha}|^2\ dz+r^{2m}\int_{Q_r(z_0)}|{\bf g}|^2\ dz\Bigg)\\
&\leq Nr^{d+2(2m-s)}\left(\sum_{k=0}^{m}\|D^k{\bf u}\|_{L_\infty(Q_{1/2})}^2+\sum_{|\alpha|\leq m}\|{\bf f}_{\alpha}\|_{L_\infty(Q_{1/2})}^2\right),
\end{align*}
where ${\bf P}:={\bf P}(x)$ is the vector-valued polynomial of order $m-1$ such that for $s=0,\dots,m-1$,
$$(D^s{\bf P})_{Q_r(z_0)}=(D^s{\bf u})_{Q_r(z_0)}.$$
Due to Campanato's characterization of H\"{o}lder continuous functions (see, for instance, \cite[Lemma 3.1]{dz2015} and \cite[Lemma 4.3]{l1996}), we derive \eqref{estclaim}.
\end{proof}

Next, we shall establish an estimate of the modulus of continuity of $(\widetilde{D}^\alpha{\bf u},{\bf U})$. See Proposition \ref{main prop} below. For this, we will localize the problem by taking $\cQ$ to be a unit cylinder $Q_{1}:=(-1,0)\times B_{1}$ and $z_{0}=(t_{0},x_{0})\in Q_{3/4}$. By suitable rotation and scaling, we may suppose that a finite number of subdomains lie in $Q_{1}$ and that they can be represented by
$$x_{d}=h_{j}(t,x'),\quad\forall~t\in(-1,0),~x'\in B'_{1},~j=1,\ldots,l(<M),$$
where
\begin{equation*}%\label{eq10.42}
-1<h_{1}(t,x')<\dots<h_{l}(t,x')<1,
\end{equation*}
$h_{j}(t,\cdot)\in C^{1,\text{Dini}}(B'_{1})$, and $h_{j}(\cdot,x')\in C^{\gamma_{0}}(-1,0)$, where $\gamma_{0}>\frac{1}{2m}$. Set $h_{0}(t,x')=-1$ and $h_{l+1}(t,x')=1$. Then we have $l+1$ regions:
\begin{equation}\label{def-subdomain}
\cQ_{j}:=\{(t,x)\in \cQ: h_{j-1}(t,x')<x_{d}<h_{j}(t,x')\},\quad1\leq j\leq l+1.
\end{equation} 
We may suppose that there exists some $\cQ_{j_{0}}$, such that $(t_{0},x_{0})\in Q_{3/4}\cap \cQ_{j_{0}}$ and the closest point on $\partial_p \cQ_{j_{0}}\cap\{t=t_0\}$ to $(t_{0},x_{0})$ is $(t_{0},x'_{0},h_{j_{0}}(t_{0},x'_{0}))$, and $\nabla_{x'}h_{j_{0}}(t_{0},x'_{0})=0'$.
We introduce the $l+1$ ``strips"
\begin{equation*}%\label{def-strip}
\Omega_{j}:=\{(t,x)\in \cQ: h_{j-1}(t_{0},x'_{0})<x_{d}<h_{j}(t_{0},x'_{0})\},\quad1\leq j\leq l+1.
\end{equation*}

As explained in the Appendix \ref{Append}, we can move the lower order terms  to the right-hand side of \eqref{systems}. Thus, in the sequel, we assume that all the lower order coefficients  are zero and consider the following parabolic systems
\begin{equation}\label{systemmain}
{\bf u}_t+(-1)^m\sum_{|\alpha|=|\beta|= m}D^{\alpha}(A^{\alpha\beta}D^{\beta}{\bf u})=\sum_{|\alpha|=m}D^{\alpha}{\bf f}_{\alpha}\quad\mbox{in}~Q_1.
\end{equation}

Recall that for each $z_0$, the coordinate system is chosen according to it. In the coordinate system associated with $z_0$, we now define
$$
\Phi(z_{0},r)=\inf_{\Theta\in\mathbb R^{n\times \tbinom{d+m-1}{m}}}\left(\fint_{Q_{r}(z_{0})}|(\widetilde{D}^\alpha{\bf u},{\bf U})-\Theta|^{\frac{1}{2}}\ dxdt\right)^{2},$$
where ${\bf U}$ is defined in \eqref{defU} with $|\beta|=m$. Let $\hat{A}^{\alpha\beta}_{(j)}\in\mathcal{A}$ be a constant function in $\cQ_{j}$ corresponding to the definition of $\omega_{A^{\alpha\beta}}(r)$ in \eqref{def omega A}. Define the piecewise constant functions
\begin{align*}
\bar{A}_{r,z_0}^{\alpha\beta}(t,x)=
\hat{A}^{\alpha\beta}_{(j)},\quad (t,x)\in\Omega_{j}.
\end{align*}
By \cite[(3.5)]{dx2021}, we have
\begin{align}\label{est A}
\fint_{Q_{r}(z_{0})}|\hat{A}^{\alpha\beta}-\bar{A}_{r,z_0}^{\alpha\beta}|\ dxdt\leq N\hat\omega_1(r),
\end{align}
where $\hat\omega_1(r)=\omega_1(r)+r^{2m\gamma_0-1}$, and $\omega_1(r)$ is a Dini function; see \cite[Lemma 3.3]{dx2021} for the detail. Using Lemmas \ref{weak est barv} and \ref{estDmw}, we shall prove the following result.

\begin{lemma}\label{lemma itera}
For any $\gamma\in (0,1)$ and $0<\rho\leq r\leq 1/4$, we have
\begin{align}\label{est phi'}
\Phi(z_{0},\rho)\leq N\Big(\frac{\rho}{r}\Big)^{\gamma}r^{-(d+2m)}\|(\widetilde{D}^\alpha{\bf u},{\bf U})\|_{L_{1}(Q_{r}(z_{0}))}
+N\tilde{\omega}_{A^{\alpha\beta}}(\rho)\|D^m{\bf u}\|_{L_{\infty}(Q_{r}(z_{0}))}+N\tilde{\omega}_{{\bf f}_{\alpha}}(\rho),
\end{align}
where $N=N(n,d,m,p,\nu,\Lambda,\gamma)$, and $\tilde\omega_{\bullet}(t)$ is a Dini function derived from $\omega_{\bullet}(t)$.
\end{lemma}

\begin{proof}
We decompose 
\begin{equation}\label{def_w}
{\bf u}={\bf u}_{1}+{\bf v}+{\bf w},
\end{equation}
where 
\begin{equation}\label{defw u1}
{\bf u}_{1}:={\bf u}_{1;r,z_0}(x_d)=(-1)^m\int_{x_{0,d}}^{x_d}\int_{x_{0,d}}^{s_{m-1}}\dots\int_{x_{0,d}}^{s_{1}}(\bar{A}_{r,z_0}^{\bar{\alpha}\bar\alpha}(s))^{-1}\bar{\bf f}_{\bar\alpha;r,z_0}(s)\,ds ds_1\dots ds_{m-1}
\end{equation}
satisfying
\begin{align*}
\partial_t{\bf u}_1+(-1)^{m}\sum_{|\alpha|=|\beta|=m}D^{\alpha}(\bar{A}_{r,z_0}^{\alpha\beta}(x_d)D^{\beta}{\bf u}_1)&=(-1)^{m}D^{\bar\alpha}(\bar{A}_{r,z_0}^{\bar\alpha\bar\alpha}(x_d)D_d^{m}{\bf u}_1)=D^{\bar\alpha}\bar{\bf f}_{\bar{\alpha};r,z_0}(x_d),
\end{align*}
with $\bar\alpha=me_d$. For $p\in(1,\infty)$, let ${\bf v}\in \mathcal{H}_{p}^{m}(Q_{r}(z_{0}))$ be the weak solution of
\begin{align}\label{sol-v}
\begin{cases}
{\bf v}_t+(-1)^{m}\tilde{\mathcal{L}_0\ }{\bf v}=\sum_{|\alpha|=m}D^{\alpha}({\bf F}_\alpha\chi_{Q_{r/2}(z_{0})})&\ \mbox{in}~Q_{r}(z_{0}),\\
{\bf v}=|D{\bf v}|=\dots=|D^{m-1}{\bf v}|=0&\ \mbox{on}~\partial_p Q_{r}(z_{0}),
\end{cases}
\end{align}
where 
\begin{equation}\label{bfF}
{\bf F}_\alpha=\sum_{|\beta|=m}(-1)^m(\bar{A}_{r,z_0}^{\alpha\beta}(x_d)-A^{\alpha\beta}(t,x))D^{\beta}{\bf u}+{\bf f}_\alpha(t,x)-\bar{{\bf f}}_{\alpha;r,z_0}(x_d).
\end{equation}
Then by using  Lemma \ref{weak est barv}, and employing the same argument as in deriving \cite[(3.7)]{dx2019}, we have
\begin{align}\label{holder v bar}
\left(\fint_{Q_{r/2}(z_{0})}|D^m{\bf v}|^{1/2}\ dxdt\right)^{2}\leq N\Big(\bar{\omega}_{A^{\alpha\beta}}(r)\|D^m{\bf u}\|_{L_{\infty}(Q_{r}(z_{0}))}+\bar\omega_{{\bf f}_{\alpha}}(r)\Big),
\end{align}
where $\bar\omega_{\bullet}(r)=\omega_{\bullet}(r)+\hat\omega_1(r)$ and $\hat\omega_1(r)$ is from \eqref{est A}. Combining \eqref{systemmain} and \eqref{bfF}, we have in $Q_{r/2}(z_{0})$, 
\begin{align}\label{eq_w}
&{\bf w}_t+(-1)^{m}\sum_{|\alpha|=|\beta|=m}D^{\alpha}(\bar{A}_{r,z_0}^{\alpha\beta}(x_d)D^{\beta}{\bf w})\nonumber\\
&=\sum_{|\alpha|=m}D^{\alpha}{\bf f}_{\alpha}+(-1)^m\sum_{|\alpha|=|\beta|=m}D^{\alpha}\big((\bar{A}_{r,z_0}^{\alpha\beta}(x_d)-A^{\alpha\beta}(t,x))D^{\beta}{\bf u}\big)-D^{\bar\alpha}\bar{\bf f}_{\bar{\alpha};r,z_0}(x_d)-\sum_{|\alpha|=m}D^{\alpha}{\bf F}_{\alpha}\nonumber\\
&=0.
\end{align}
Moreover, $\widetilde{D}^\alpha{\bf u}=\widetilde{D}^\alpha{\bf w}+\widetilde{D}^\alpha{\bf v}$, and
\begin{align*}
{\bf U}&=\sum_{|\beta|=m}A^{\bar{\alpha}\beta}D^{\beta}{\bf u}-(-1)^m{\bf f}_{\bar{\alpha}}\\
&=\sum_{|\beta|=m}\big(A^{\bar{\alpha}\beta}-A_{r,z_0}^{\bar{\alpha}\beta}(x_d)\big)D^{\beta}{\bf u}+{\bf W}+(-1)^m\bar{\bf f}_{\bar{\alpha};r,z_0}(x_d)+{\bf V}-(-1)^m{\bf f}_{\bar{\alpha}},
\end{align*}
where 
\begin{equation}\label{def_V-W}
{\bf V}=\sum_{|\beta|=m}\bar{A}_{r,z_0}^{\bar{\alpha}\beta}(x_d)D^{\beta}{\bf v},\quad{\bf W}=\sum_{|\beta|=m}\bar{A}_{r,z_0}^{\bar{\alpha}\beta}(x_d)D^{\beta}{\bf w}.
\end{equation}
Therefore, in view of the triangle inequality, \eqref{holder w bar}, and \eqref{holder v bar}, we obtain
\begin{align*}
&\left(\fint_{Q_{\kappa r}(z_{0})}|\widetilde{D}^\alpha{\bf u}-(\widetilde{D}^\alpha{\bf w})_{Q_{\kappa r}(z_{0})}|^{1/2}\ dxdt+\fint_{Q_{\kappa r}(z_{0})}|{\bf U}-({\bf W})_{Q_{\kappa r}(z_{0})}|^{1/2}\ dxdt\right)^{2}\nonumber\\
&\leq N\kappa\left(\fint_{Q_{r/2}(z_{0})}|(\widetilde{D}^\alpha{\bf w},{\bf W})-\Theta|^{1/2}\ dxdt\right)^{2}\nonumber\\
&\quad+N\kappa^{-2(d+2m)}\Big(\bar{\omega}_{A^{\alpha\beta}}(r)\|D^m{\bf u}\|_{L_{\infty}(Q_{r}(z_{0}))}+\bar\omega_{{\bf f}_{\alpha}}(r)\Big)\\
&\leq N_0\kappa\left(\fint_{Q_{r/2}(z_{0})}|(\widetilde{D}^\alpha{\bf u},{\bf U})-\Theta|^{1/2}\ dx\ dt\right)^{2}+N\kappa^{-2(d+2m)}\Big(\bar{\omega}_{A^{\alpha\beta}}(r)\|D^m{\bf u}\|_{L_{\infty}(Q_{r}(z_{0}))}\nonumber\\
&\quad+\bar\omega_{{\bf f}_{\alpha}}(r)\Big),
\end{align*}
where $N_0$ depends on $n,d,p,\nu$, and $\Lambda$. Since $\Theta\in \mathbb R^{n\times \tbinom{d+m-1}{m}}$ is arbitrary, we obtain
\begin{align*}%\label{iterating Du}
\Phi(z_{0},\kappa r)\leq N_{0}\kappa\Phi(z_{0},r)+N\kappa^{-2(d+2m)}\Big(\|D^m{\bf u}\|_{L_{\infty}(Q_{r}(z_{0}))}\bar{\omega}_{A^{\alpha\beta}}(r)+\bar\omega_{{\bf f}_\alpha}(r)\Big).
\end{align*}
For any given $\gamma\in(0,1)$, fix a $\kappa\in(0,1/2)$ sufficiently small such that $N_{0}\kappa\leq\kappa^{\gamma}$. This gives 
\begin{align*}
\Phi(z_{0},\kappa r)\leq \kappa^{\gamma}\Phi(z_{0},r)+N\Big(\|D^m{\bf u}\|_{L_{\infty}(Q_{r}(z_{0}))}\bar{\omega}_{A^{\alpha\beta}}(r)+\bar\omega_{{\bf f}_\alpha}(r)\Big).
\end{align*}
By iteration and $\kappa^{\gamma}<1$, we obtain for $j=1,2,\ldots$,
\begin{align*}
\Phi(z_{0},\kappa^{j}r)
&\leq\kappa^{j\gamma}\Phi(z_{0},r)+N\Big(\|D^m{\bf u}\|_{L^{\infty}(Q_{r}(z_{0}))}\sum_{i=1}^{j}\kappa^{(i-1)\gamma}\bar{\omega}_{A^{\alpha\beta}}(\kappa^{j-i}r)\nonumber\\
&\quad+\sum_{i=1}^{j}\kappa^{(i-1)\gamma}\bar\omega_{{\bf f}_\alpha}(\kappa^{j-i}r)\Big),
\end{align*}
which implies 
\begin{align}\label{iteration}
\Phi(z_{0},\kappa^{j}r)\leq\kappa^{j\gamma}\Phi(z_{0},r)+N\|D^m{\bf u}\|_{L_{\infty}(Q_{r}(z_{0}))}\tilde{\omega}_{A^{\alpha\beta}}(\kappa^{j}r)+N\tilde\omega_{{\bf f}_\alpha}(\kappa^{j}r),
\end{align}
where 
\begin{align*}%\label{tilde phi}
\tilde\omega_{\bullet}(t)=\sum_{i=1}^{\infty}\kappa^{i\gamma}\Big(\bar\omega_{\bullet}(\kappa^{-i}t)\chi_{\kappa^{-i}t\leq1}+\bar\omega_{\bullet}(1)\chi_{\kappa^{-i}t>1}\Big).
\end{align*}
Therefore, for any $\rho$ with $0<\rho\leq r\leq1/4$ and $\kappa^j r\leq\rho<\kappa^{j-1}r$, we have \eqref{est phi'} and  the lemma is proved.
\end{proof}

Using Lemma \ref{lemma itera}, 
\begin{align*}%\label{est-Ddmu}
D_d^m{\bf u}=\big(A^{\bar{\alpha}\bar\alpha}\big)^{-1}\left({\bf U}-\sum_{|\beta|=m,~\beta\neq me_d}A^{\bar{\alpha}\beta}D^{\beta}{\bf u}+(-1)^m{\bf f}_{\bar{\alpha}}\right),
\end{align*}
and the same argument as in \cite[Lemma 3.4]{dx2019}, under the assumption that $D^m{\bf u}$ is locally bounded, we derive an a priori $L_\infty$-estimate for $D^m{\bf u}$ as follows. 
\begin{lemma}\label{Dmubdd}
There is a constant $N>0$  depending on $n,d,m,p,\nu,\Lambda$, and $\omega_A$, such that
\begin{align}\label{est Dmu''}
\|D^m{\bf u}\|_{L_\infty(Q_{1/4})}\leq N\|(\widetilde{D}^\alpha{\bf u},{\bf U})\|_{L_{1}(Q_{1})}+N\left(\int_0^1\frac{\tilde{\omega}_{{\bf f}_{\alpha}}(s)}{s}\ ds+\|{\bf f}_{\bar{\alpha}}\|_{L_\infty(Q_1)}\right).
\end{align}
\end{lemma}

For any point $z_1$ in the same subdomain as $z_0$, we denote 
$$
\Phi_{z_0}(z_{1},r)=\inf_{\Theta\in\mathbb R^{n\times \tbinom{d+m-1}{m}}}\left(\fint_{Q_{r}(z_{1})}|(\widetilde{D}^\alpha{\bf u},{\bf U})-\Theta|^{\frac{1}{2}}\ dxdt\right)^{2},$$
where $\Phi_{z_0}(z_{1},r)$ means that the function is defined in the coordinate system associated with $z_0$.

\begin{lemma}\label{lemma itera2}
Let $z_0$ and $z_1$ be two points in the same subdomain. Then for any $\gamma\in (0,1)$ and $0<\rho\leq r\leq 1/4$, we have
\begin{align}\label{est phi'2}
\Phi_{z_0}(z_{1},\rho)&\leq N\Big(\frac{\rho}{r}\Big)^{\gamma}r^{-(d+2m)}\|(\widetilde{D}^\alpha{\bf u},{\bf U})\|_{L_{1}(Q_{r}(z_{1}))}
+N\tilde{\omega}_{A^{\alpha\beta}}(\rho)\|D^m{\bf u}\|_{L_{\infty}(Q_{r}(z_{1}))}+N\tilde{\omega}_{{\bf f}_{\alpha}}(\rho),
\end{align}
where $Q_{r}(z_{1})\subset\cQ_{j}$ for some $j=1,\dots,l+1$, $N=N(n,d,m,p,\nu,\Lambda,\gamma)$, and $\tilde\omega_{\bullet}(t)$ is a Dini function derived from $\omega_{\bullet}(t)$.
\end{lemma}

\begin{proof}
Let $\hat{A}^{\bar{\alpha}\beta}$ and $\hat{{\bf f}}_{\bar{\alpha}}$ be constants corresponding to $A^{\bar{\alpha}\beta}$ and ${\bf f}_{\bar{\alpha}}$, respectively, where $\bar{\alpha}=me_d$. 
For any $\Theta=(\widetilde\Theta_{\alpha},\Theta_{\bar{\alpha}})$, $|\alpha|=m$, and $\alpha\neq me_d$, using $Q_{r}(z_{1})\subset\cQ_{j}$ and the triangle inequality, we obtain 
\begin{align}\label{estPhiz1}
&\Phi_{z_{0}}(z_1,r)\leq\left(\fint_{Q_{r}(z_1)}\Big|\big(\widetilde{D}^\alpha{\bf u}-\widetilde\Theta_\alpha ,{\bf U}+(-1)^m\hat{{\bf f}}_{\bar{\alpha}}-\sum_{|\beta|=m}\hat{A}^{\bar{\alpha}\beta}\Theta_{\beta}\big)\Big|^{1/2}\right)^{2}\nonumber\nonumber\\
&=\Bigg(\fint_{Q_{r}(z_1)}|(\widetilde{D}^\alpha{\bf u}-\widetilde\Theta_\alpha,\sum_{|\beta|=m}\hat{A}^{\bar{\alpha}\beta}D^{\beta}{\bf u}-\sum_{|\beta|=m}\hat{A}^{\bar{\alpha}\beta}\Theta_{\beta})\nonumber\\
&\quad+(0',{\bf U}-\sum_{|\beta|=m}\hat{A}^{\bar{\alpha}\beta}D^{\beta}{\bf u}+(-1)^m\hat{{\bf f}}_{\bar{\alpha}})|^{1/2}\Bigg)^{2}\nonumber\\
&=\left(\fint_{Q_{r}(z_1)}|(\widetilde{D}^\alpha{\bf u}-\widetilde\Theta_\alpha,D_{d}^m{\bf u}-\Theta_{\bar\alpha})\mathcal{M}+(0',{\bf U}-\sum_{|\beta|=m}\hat{A}^{\bar{\alpha}\beta}D^{\beta}{\bf u}+(-1)^m\hat{{\bf f}}_{\bar{\alpha}})|^{1/2}\right)^{2}\nonumber\\
&\leq N\Psi(z_1,r)+N\left({\omega}_{A^{\alpha\beta}}(r)\|D^m{\bf u}\|_{L_{\infty}(Q_{r}(z_{1}))}
+\omega_{{\bf f}_{\bar{\alpha}}}(r)\right),
\end{align}
where $\widetilde{D}^\alpha{\bf u}$ is the collection of $D^\alpha{\bf u}$, $|\alpha|=m$, $\alpha\neq me_d$, $\mathcal{M}$ is defined by \eqref{DEF_M} with $\hat{A}^{\bar{\alpha}\beta_i}$ in place of $\bar{A}^{\bar{\alpha}\beta_i}(x_d)$, and
$$\Psi(z_1,r):=\inf_{\Theta\in\mathbb R^{n\times \tbinom{d+m-1}{m}}}\left(\fint_{Q_{r}(z_1)}|D^m{\bf u}-\Theta |^{1/2}\ dxdt\right)^{2},$$
which is independent of coordinate systems. 

Next we estimate $\Psi(z_1,r)$ in the coordinate system associated with $z_1$. Note that
\begin{align*}%\label{identity1}
&(\widetilde{D}^\alpha{\bf u}-\widetilde\Theta_\alpha ,\sum_{|\beta|=m}\hat{A}^{\bar{\alpha}\beta}D^{\beta}{\bf u}-(-1)^m\hat{{\bf f}}_{\bar{\alpha}}-\Theta_{\bar{\alpha}})\nonumber\\
&=\Big(\widetilde{D}^\alpha{\bf u}-\widetilde\Theta_\alpha,D_{d}^m{\bf u}-(\hat{A}^{\bar{\alpha}\bar\alpha})^{-1}\big((-1)^m\hat{{\bf f}}_{\bar{\alpha}}
+\Theta_{\bar{\alpha}}-\sum_{\substack{|\beta|=m\\
\beta\neq me_d}}\hat{A}^{\bar{\alpha}\beta}\Theta_{\beta}\big)\Big)\mathcal{M}.
\end{align*}
Here, $(\hat{A}^{\bar{\alpha}\bar\alpha})^{-1}\big((-1)^m\hat{{\bf f}}_{\bar{\alpha}}
	+\Theta_{\bar{\alpha}}-\sum_{\substack{|\beta|=m\\
			\beta\neq me_d}}\hat{A}^{\bar{\alpha}\beta}\Theta_{\beta}\big)$ 
is a (vector-valued) constant in $Q_{r}(z_1)\subset\cQ_j$ for some $j=1,\dots,l+1$. Then by the triangle inequality, we have 
\begin{align*}%\label{varphy1}
&\Psi(z_1,r)\leq \left(\fint_{Q_{r}(z_1)}\Big|\big(\widetilde{D}^\alpha{\bf u}-\widetilde\Theta_\alpha,D_{d}^m{\bf u}-(\hat{A}^{\bar{\alpha}\bar\alpha})^{-1}\big((-1)^m\hat{{\bf f}}_{\bar{\alpha}}
+\Theta_{\bar{\alpha}}-\sum_{\substack{|\beta|=m\\
\beta\neq me_d}}\hat{A}^{\bar{\alpha}\beta}\Theta_{\beta}\big)\big)\Big|^{1/2}\right)^{2}\nonumber\\
&=\left(\fint_{Q_{r}(z_1)}|(\widetilde{D}^\alpha{\bf u}-\widetilde\Theta_\alpha ,\sum_{|\beta|=m}\hat{A}^{\bar{\alpha}\beta}D^{\beta}{\bf u}-(-1)^m\hat{{\bf f}}_{\bar{\alpha}}-\Theta_{\bar{\alpha}})\mathcal{M}^{-1}|^{1/2}\right)^{2}\nonumber\\
&=\Bigg(\fint_{Q_{r}(z_1)}\Big|\big((\widetilde{D}^\alpha{\bf u}-\widetilde\Theta_\alpha,{\bf U}-\Theta_{\bar{\alpha}})\nonumber\\
&\qquad+(0',\sum_{|\beta|=m}\hat{A}^{\bar{\alpha}\beta}D^{\beta}{\bf u}-\sum_{|\beta|=m}A^{\bar{\alpha}\beta}D^{\beta}{\bf u}+(-1)^m{\bf f}_{\bar{\alpha}}-(-1)^m\hat{{\bf f}}_{\bar{\alpha}})\big)\mathcal{M}^{-1}\Big|^{1/2}\Bigg)^{2}\nonumber\\
&\leq N\left(\fint_{Q_{r}(z_1)}|(\widetilde{D}^\alpha{\bf u},{\bf U})-\Theta|^{1/2}\right)^{2}+N\left({\omega}_{A^{\alpha\beta}}(r)\|D^m{\bf u}\|_{L_{\infty}(Q_{r}(z_{1}))}
+\omega_{{\bf f}_{\bar{\alpha}}}(r)\right).
\end{align*}
Since $\Theta$ is arbitrary, we obtain
\begin{align*}
\Psi(z_1,r)\leq N\Phi(z_{1},r)+N\left({\omega}_{A^{\alpha\beta}}(r)\|D^m{\bf u}\|_{L_{\infty}(Q_{r}(z_{1}))}
+\omega_{{\bf f}_{\bar{\alpha}}}(r)\right),
\end{align*}
where $\Phi(z_{1},r)$ is the function is defined in the coordinate system associated with $z_1$. Substituting it into \eqref{estPhiz1}, we derive 
\begin{equation*}
\Phi_{z_{0}}(z_1,r)\leq N\Phi(z_{1},r)+N\left({\omega}_{A^{\alpha\beta}}(r)\|D^m{\bf u}\|_{L_{\infty}(Q_{r}(z_{1}))}
+\omega_{{\bf f}_{\bar{\alpha}}}(r)\right).
\end{equation*}
By using \eqref{est phi'} with $z_1$ in place of $z_0$, we get \eqref{est phi'2}.
\end{proof}

With Lemmas \ref{lemma itera} and \ref{Dmubdd} in hand, we are ready to prove the following proposition. Together with Lemma \ref{lemmaC12} and the results in Appendix \ref{Append}, Theorem \ref{mainthm} follows.

\begin{prop}\label{main prop}
Let $\varepsilon, \gamma\in(0,1)$ and $p\in(1,\infty)$. Suppose that $A^{\alpha\beta}$ and ${\bf f}_\alpha$ are of piecewise Dini mean oscillation in $Q_{1}$, and ${\bf f}_{\alpha}\in L_\infty(Q_1)$. If ${\bf u}\in \mathcal{H}_{p}^{m}(Q_{1})$ is a weak solution to
$${\bf u}_{t}+(-1)^{m}\sum_{|\alpha|=|\beta|= m}D^{\alpha}(A^{\alpha\beta}D^{\beta}{\bf u})=\sum_{|\alpha|= m}D^{\alpha}{\bf f}_{\alpha}\quad\mbox{in}~ Q_{1},$$
then ${\bf u}\in C^{1/2,m}(\overline{{\cQ}_{j}}\cap ((-1+\varepsilon,0)\times B_{1-\varepsilon}))$, $j=1,\ldots,M$, and 
\begin{equation*}
\|D^m{\bf u}\|_{L_\infty((-1+\varepsilon,0)\times B_{1-\varepsilon})}\leq N\|(\widetilde{D}^\alpha{\bf u},{\bf U})\|_{L_{1}(Q_{1})}+N\left(\int_0^1\frac{\tilde{\omega}_{{\bf f}_{\alpha}}(s)}{s}\ ds+\|{\bf f}_{\bar{\alpha}}\|_{L_\infty(Q_1)}\right).
\end{equation*}
Furthermore, for any fixed $z_{0}\in (-1+\varepsilon,0)\times B_{1-\varepsilon}$, there exists a coordinate system associated with $z_{0}$, such that for all $z_{1}\in (-1+\varepsilon,0)\times B_{1-\varepsilon}$, we have
\begin{align*}
&|(\widetilde{D}^\alpha{\bf u}(z_0),{\bf U}(z_0))-(\widetilde{D}^\alpha{\bf u}(z_1),{\bf U}(z_1))|\\
&\leq N|z_{0}-z_{1}|_{p}^{\gamma}\|(\widetilde{D}^\alpha{\bf u},{\bf U})\|_{L_{1}(Q_{1})}+N\int_{0}^{|z_{0}-z_{1}|_{p}}\frac{\tilde{\omega}_{{\bf f}_\alpha}(s)}{s}\ ds\\
&\quad+N\int_{0}^{|z_{0}-z_{1}|_{p}}\frac{\tilde{\omega}_{A^{\alpha\beta}}(s)}{s}\ ds\cdot\left(\|(\widetilde{D}^\alpha{\bf u},{\bf U})\|_{L_{1}(Q_{1})}+\int_{0}^{1}\frac{\tilde{\omega}_{{\bf f}_\alpha}(s)}{s}\ ds+\|{\bf f}_\alpha\|_{L_{\infty}(Q_{1})}\right),
\end{align*}
where $\widetilde{D}^\alpha{\bf u}$ is the collection of $D^\alpha{\bf u}$, $\alpha=m,~\alpha\neq me_d$, ${\bf U}$ is defined in \eqref{defU} with $|\beta|=m$, $N$ depends on $n,d,m,M,p,\Lambda,\nu,\varepsilon,\gamma$, and the $C^{1,\text{Dini}}$ and $C^{\gamma_{0}}$ characteristics of $\cQ_{j}$ with respect to $x$ and $t$, respectively, and $\tilde\omega_{\bullet}(t)$ is a Dini function derived from $\omega_{\bullet}(t)$. 
\end{prop}

\begin{proof}
We shall divide the proof of Proposition \ref{main prop} into two steps. We temporarily assume that $D^m{\bf u}$ is locally bounded with the estimate \eqref{est Dmu''} in Step 1. Then we drop the temporary assumption in Step 2.

{\bf Step 1}. In this step, we will prove Proposition \ref{main prop} by assuming that $D^m{\bf u}$ is locally bounded. Take $\Theta_{z_0,r}\in\mathbb R^{n\times \tbinom{d+m-1}{m}}$ such that
\begin{align*}
\Phi(z_{0},r)=\left(\fint_{Q_{r}(z_{0})}|(\widetilde{D}^\alpha{\bf u},{\bf U})-\Theta_{z_0,r}|^{\frac{1}{2}}\ dxdt\right)^{2}.
\end{align*}
We can find $\Theta_{z_0,\kappa^{l}r}$ in the same way, $l=1,\dots$. Using \eqref{iteration}, we have 
\begin{equation*}
\lim_{l\rightarrow\infty}\Phi(z_{0},\kappa^{l}r)=0.
\end{equation*}
Then for a.e. $z_0\in Q_{1/8}$, we have
\begin{equation}\label{limPhi}
\lim_{l\rightarrow\infty}\Theta_{z_0,\kappa^{l}r}=(\widetilde{D}^\alpha{\bf u}(z_0),{\bf U}(z_0)).
\end{equation}
It follows from the triangle inequality  that,  for $L\geq1$,
\begin{align}\label{estTheta}
|\Theta_{z_0,\kappa^{L}r}-\Theta_{z_0,r}|\leq N\sum_{l=0}^{L}\Phi(z_{0},\kappa^{l}r).
\end{align}
Thus, taking $L\rightarrow\infty$ in \eqref{estTheta}, combining \eqref{limPhi}, \eqref{iteration}, the fact that
$$\sum_{i=0}^{\infty}\tilde\omega_{\bullet}(\kappa^{i}r)\leq N\int_{0}^{r}\frac{\tilde\omega_{\bullet}(s)}{s}\ ds,$$
(see, for instance, \cite[Lemma 2.7]{dk2017}), and \eqref{est phi'}, we obtain for $0<r<1/8$,
\begin{align}\label{Duz0}
&|(\widetilde{D}^\alpha{\bf u}(z_0),{\bf U}(z_0))-\Theta_{z_0,r}|\nonumber\\
&\leq N\sum_{l=0}^{\infty}\Phi(z_{0},\kappa^{l}r)\nonumber\\
&\leq N\left(\Phi(z_{0},r)
+\|D^m{\bf u}\|_{L_{\infty}(Q_{1/4})}\int_0^r\frac{\tilde{\omega}_{A^{\alpha\beta}}(s)}{s}\ ds+\int_0^r\frac{\tilde{\omega}_{{\bf f}_\alpha}(s)}{s}\ ds\right)\nonumber\\
&\leq Nr^\gamma\|(\widetilde{D}^\alpha{\bf u},{\bf U})\|_{L_{1}(Q_{1/4})}+N\left(\|D^m{\bf u}\|_{L_{\infty}(Q_{1/4})}\int_0^r\frac{\tilde{\omega}_{A^{\alpha\beta}}(s)}{s}\ ds+\int_0^r\frac{\tilde{\omega}_{{\bf f}_\alpha}(s)}{s}\ ds\right).
\end{align}
Suppose that $z_{1}=(t_1,x_1)\in Q_{1/8}\cap \cQ_{j_{1}}$ for some $j_{1}\in[1,l+1]$. The case of $|z_{0}-z_{1}|_{p}\geq 1/32$ follows from 
\begin{align*}
&|(\widetilde{D}^{\alpha}{\bf u}(z_{0}),{\bf U}(z_{0}))-(\widetilde{D}^{\alpha}{\bf u}(z_{1}),{\bf U}(z_{1}))|\leq N|z_{0}-z_{1}|_{p}^\gamma\big(\|D^m{\bf u}\|_{L_\infty(Q_{1/4})}+\|{\bf f}_{\bar\alpha}\|_{L_\infty(Q_{1/4})}\big)
\end{align*}
and Lemma \ref{Dmubdd}.  When  $|z_{0}-z_{1}|_{p}\leq 1/32$, we set $r=|z_{0}-z_{1}|_{p}$. Recalling that for each $z_0$, the coordinate system is chosen according to it. Now let us continue the proof according to the following two cases:

{\bf Case 1}: $Q_{r}(z_1) \subset \cQ_{j_0}$. By Lemma \ref{lemma itera2} and a similar argument that led to \eqref{Duz0}, we obtain
\begin{align}\label{Duz1}
&|(\widetilde{D}^\alpha{\bf u}(z_1),{\bf U}(z_1))-\Theta_{z_1,r}|\nonumber\\
&\leq N\sum_{l=0}^{\infty}\Phi_{z_0}(z_{1},\kappa^{l}r)\nonumber\\
&\leq Nr^\gamma\|(\widetilde{D}^\alpha{\bf u},{\bf U})\|_{L_{1}(Q_{1/4})}+N\left(\|D^m{\bf u}\|_{L_{\infty}(Q_{1/4})}\int_0^r\frac{\tilde{\omega}_{A^{\alpha\beta}}(s)}{s}\ ds+\int_0^r\frac{\tilde{\omega}_{{\bf f}_\alpha}(s)}{s}\ ds\right).
\end{align}
By using the triangle inequality, we have
\begin{align*}%\label{case1}
&|(\widetilde{D}^{\alpha}{\bf u}(z_{0}),{\bf U}(z_{0}))-(\widetilde{D}^{\alpha}{\bf u}(z_{1}),{\bf U}(z_{1}))|^{1/2}\nonumber\\
&\leq|(\widetilde{D}^{\alpha}{\bf u}(z_{0}),{\bf U}(z_{0}))-\Theta _{z_{0},r}|^{1/2}+|(\widetilde{D}^{\alpha}{\bf u}(z),{\bf U}(z))-\Theta_{z_{0},r}|^{1/2}+|(\widetilde{D}^{\alpha}{\bf u}(z),{\bf U}(z))-\Theta _{z_1,r}|^{1/2}\nonumber\\
&\quad+|(\widetilde{D}^{\alpha}{\bf u}(z_{1}),{\bf U}(z_{1}))-\Theta_{z_1,r}|^{1/2},\quad\forall~z\in Q_{r}(z_{0})\cap Q_{r}(z_1).
\end{align*}
Then taking the average over $z\in Q_{r}(z_{0})\cap Q_{r}(z_1)$ and taking the square, using \eqref{Duz0}, \eqref{Duz1}, and Lemma \ref{Dmubdd}, we derive 
\begin{align}\label{est-prori}
&|(\widetilde{D}^{\alpha}{\bf u}(z_{0}),{\bf U}(z_{0}))-(\widetilde{D}^{\alpha}{\bf u}(z_{1}),{\bf U}(z_{1}))|\nonumber\\
&\leq N|z_{0}-z_{1}|_{p}^\gamma\|(\widetilde{D}^\alpha{\bf u},{\bf U})\|_{L_{1}(Q_{1/4})}+N\int_0^r\frac{\tilde{\omega}_{{\bf f}_\alpha}(s)}{s}\ ds\nonumber\\
&\quad+N\int_0^r\frac{\tilde{\omega}_{A^{\alpha\beta}}(s)}{s}\ ds\left(\|(\widetilde{D}^\alpha{\bf u},{\bf U})\|_{L_{1}(Q_{1})}+\int_0^1\frac{\tilde{\omega}_{{\bf f}_{\alpha}}(s)}{s}\ ds+\|{\bf f}_{\bar{\alpha}}\|_{L_\infty(Q_1)}\right).
\end{align}

{\bf Case 2}:  $Q_{r}(z_1) \not\subset \cQ_{j_0}$.
By a similar argument as in \cite[pp. 21-22]{dx2021}, we obtain \eqref{est-prori}. This finishes the proof of Proposition \ref{main prop} under the assumption that $D^m{\bf u}$ is locally bounded.

{\bf Step 2}. Now let us remove the assumption that $D^m{\bf u}$ is locally bounded by using the technique of flattening the boundary and an approximation argument.
For any point $z_0=(t_0,x_0)\in\partial_p\cQ_{j_0}$, $j_0=1,\dots,M-1$, we will prove that $D^m{\bf u}$ is bounded in a neighborhood of $z_0$. Since $z_0$ belongs to the boundaries of at most two of the subdomains, we may assume that there exist a constant $0<r<1$ and a function $h_{j_0}(t,x')$ which is $C^{1,\text{Dini}}$ in $x'$, $C^{\gamma_{0}}$ in $t$, and $|\nabla_{x'}h_{j_0}(t_0,x'_0)|=0$, such that $h_{j_0}(t,x')$ divides $Q_{r}(z_0)$ into two subdomains $Q_{r}^+(z_0)$ and $Q_{r}^-(z_0)$. 
Let $\zeta_0$ be a  smooth even function in $\mathbb R$ with a compact support in $(-1,1)$ satisfying 
\begin{equation*}
\int_{\mathbb R}\zeta_0(t)\ dt=1,\quad \int_{\mathbb R}t^2\zeta_0(t)\ dt=0.
\end{equation*}
Set 
\begin{equation*}
\zeta(t,x')=\zeta_0(t+1)\prod_{i=1}^{d-1}\zeta_0(x_i).
\end{equation*}
For $\varepsilon>0$, let $$\zeta_{\varepsilon}(t,x')=\varepsilon^{-(d+2m-1)}\zeta(\varepsilon^{-2m}t,\varepsilon^{-1}x').$$ 
Let $h_{j_0,\varepsilon}(t,x')$ be a mollification of $h_{j_0}(t,x')$ defined as follows:
\begin{align*}
h_{j_0,\varepsilon}(t,x')=\int_{\mathbb R}\int_{\mathbb R^{d-1}}h_{j_0}(s,y')\zeta_{\varepsilon}(t-s,x'-y')\ dy'ds,
\end{align*} 
where $\varepsilon>0$ is a sufficiently small constant.

Define the shifted point
\begin{equation*}
x_\varepsilon=\begin{cases}
x+\lambda\varepsilon e_d,\quad x_d>h_{j_0,\varepsilon}(t,x'),\\
x-\lambda\varepsilon e_d,\quad x_d<h_{j_0,\varepsilon}(t,x'),
\end{cases}
\end{equation*}
where $\lambda>0$ is some number such that $Q_{\varepsilon}(t_0,x_\varepsilon):=(t_0-\varepsilon^{2m},t_0)\times B_{\varepsilon}(x_\varepsilon)$ lies in $Q_{r}^+(z_0)$ or $Q_{r}^-(z_0)$. Set 
\begin{equation*}
\eta_\varepsilon(t,x)=\varepsilon^{-(d+2m)}\eta(\varepsilon^{-2m}t,\varepsilon^{-1}x),
\end{equation*}
where $\eta\geq0$ is a infinitely differentiable function with unit integral supported in $Q_1$. Define the piecewise mollification of $A^{\alpha\beta}$ as follows:
\begin{equation*}
A^{\alpha\beta}_\varepsilon(t,x)=\int_{Q_\varepsilon(t_0,x_\varepsilon)}\eta_\varepsilon(t-s,x_\varepsilon-y)A^{\alpha\beta}(s,y)\ ds dy.
\end{equation*}
Similarly, we can define the  piecewise mollification ${\bf f}_{\alpha,\varepsilon}$ of ${\bf f}_{\alpha}$. Then $A^{\alpha\beta}_\varepsilon$ satisfies \eqref{ellipticity}, and  $A^{\alpha\beta}_\varepsilon$, ${\bf f}_{\alpha,\varepsilon}$ are of piecewise smooth in $Q_r(z_0)$ satisfying
$$\omega_{A^{\alpha\beta}_\varepsilon}\leq \omega_{A^{\alpha\beta}},\quad \omega_{{\bf f}_{\alpha,\varepsilon}}\leq \omega_{{\bf f}_{\alpha}}.$$
Moreover, as $\varepsilon\rightarrow 0^+$, 
\begin{equation*}
{\bf f}_{\alpha,\varepsilon}\rightarrow{\bf f}_{\alpha}~\mbox{in}~L_2(Q_r(z_0)),\quad A^{\alpha\beta}_\varepsilon\rightarrow A^{\alpha\beta}~a.e.
\end{equation*}

Let ${\bf v}_\varepsilon\in \mathcal{H}_p^m(Q_r(z_0))$ $(1<p<\infty)$ be the solution to 
\begin{equation}\label{sys-v}
\begin{cases}
\partial_t{\bf v}_\varepsilon+(-1)^m\sum_{|\alpha|=|\beta|=m}D_\alpha(A^{\alpha\beta}_{\varepsilon}D^\beta {\bf v}_\varepsilon)=\sum_{|\alpha|=m}D_\alpha ({\bf f}_\alpha-{\bf f}_{\alpha, \varepsilon})\\
\qquad\qquad\qquad\qquad\qquad\qquad+(-1)^m\sum_{|\alpha|=|\beta|=m}D_\alpha(({A}^{\alpha\beta}_{\varepsilon}-A^{\alpha\beta}) D^\beta {\bf u})\quad\mbox{in}~Q_r(z_0),\\
{\bf v}_\varepsilon=|D{\bf v}_\varepsilon|=\cdots=|D^{m-1}{\bf v}_\varepsilon|=0\quad\mbox{on}~\partial_pQ_r(z_0).
\end{cases}
\end{equation}
Note that the right-hand side of the equation in \eqref{sys-v} goes to zero as $\varepsilon\rightarrow 0^+$. Then by the $\mathcal{H}_2^m$-estimate, we have 
\begin{equation*}
\|D^m{\bf v}_\varepsilon\|_{L_2(Q_r(z_0))}\rightarrow0\quad\mbox{as}~\varepsilon\rightarrow 0^+.
\end{equation*}
This implies that there is a subsequence, still denoted by ${\bf v}_\varepsilon$, such that $D^m{\bf v}_\varepsilon\rightarrow0,~a.e.$ in $Q_r(z_0)$. 
Set ${\bf u}_\varepsilon={\bf u}-{\bf v}_\varepsilon$. Then in $Q_r(z_0)$,
\begin{equation}\label{eq-u}
\partial_t{\bf u}_\varepsilon+(-1)^m\sum_{|\alpha|=|\beta|=m}D_\alpha(A^{\alpha\beta}_{\varepsilon}D^\beta {\bf u}_\varepsilon)=\sum_{|\alpha|=m}D_\alpha{\bf f}_{\alpha,\varepsilon},
\end{equation}
and hence, to show the boundedness of $D^m{\bf u}$ near $z_0$, it suffices to prove that $D^m{\bf u}_\varepsilon$ is uniformly bounded near $z_0$. To this end, let 
\begin{equation*}
(s,y',y_d)=\boldsymbol\Phi(t,x',x_d)=(t,x',x_d-h_{j_0,\varepsilon}(t,x')),
\end{equation*}
and define $\tilde{\bf u}_\varepsilon(s,y',y_d)={\bf u}_\varepsilon(t,x',x_d)$. Then from \eqref{eq-u}, we have 
\begin{equation}\label{eq-tildeu}
\partial_t\tilde{\bf u}_\varepsilon+(-1)^m\sum_{|\alpha|=|\beta|=m}D_\alpha(\tilde A^{\alpha\beta}_{\varepsilon}D^\beta \tilde{\bf u}_\varepsilon)-\partial_th_{j_0,\varepsilon}D_d\tilde{\bf u}_\varepsilon=\sum_{|\alpha|=m}D_\alpha\tilde{\bf f}_{\alpha,\varepsilon}\quad\mbox{in}~Q_{r_1},
\end{equation}
where $\overline{Q_{r_1}}\subset\boldsymbol\Phi(Q_r(z_0))$, $\tilde A^{\alpha\beta}_\varepsilon$ and $\tilde{\bf f}_{\alpha,\varepsilon}$ are the new coefficients and data under the transformation $\Lambda:=(\frac{\partial y_i}{\partial x_j})_{i,j=1}^{d}$, which are piecewise DMO in $Q_{r_1}$. To prove $D^m{\tilde{\bf u}_\varepsilon}$ is locally bounded, we first note that, by using Lemma \ref{lemlocal}, we have for any $p\in(1,\infty)$,
\begin{align}\label{est-Dmu}
\|D^m\tilde{\bf u}_\varepsilon\|_{L_{p}(Q_{{3r_1}/4})}\leq N\big(\|\tilde{\bf u}_\varepsilon\|_{L_{p}(Q_{r_1})}+\sum_{|\alpha|=m}\|\tilde{\bf f}_{\alpha,\varepsilon}\|_{L_{p}(Q_{r_1})}\big),
\end{align}
where the constant $N$ depends on $n,d,m,p,\nu,\Lambda,r_0$, and $\|\partial_th_{j_0,\varepsilon}\|_{L_\infty}$. Following the same argument as in \cite[Lemma 3.5]{dk2011}, we have 
\begin{align}\label{partiat}
\|\partial_t\tilde{\bf u}_\varepsilon\|_{L_{2}(Q_{{2r_1}/3})}\leq N\big(\|D^m\tilde{\bf u}_\varepsilon\|_{L_{2}(Q_{{3r_1}/4})}+\sum_{|\alpha|=m}\|\tilde{\bf f}_{\alpha,\varepsilon}\|_{L_{2}(Q_{r_1})}\big).
\end{align}
Next by taking the derivative of the equation \eqref{eq-tildeu} with respect to $t$, we have 
\begin{align*}
&\partial_t(\partial_t\tilde{\bf u}_\varepsilon)+(-1)^m\sum_{|\alpha|=|\beta|=m}D_\alpha(\tilde A^{\alpha\beta}_{\varepsilon}D^\beta \partial_t\tilde{\bf u}_\varepsilon)-\partial_th_{j_0,\varepsilon}D_d\partial_t\tilde{\bf u}_\varepsilon \\
&=\sum_{|\alpha|=m}D_\alpha(\partial_t\tilde{\bf f}_{\alpha,\varepsilon})+\partial_t^2h_{j_0,\varepsilon}D_d\tilde{\bf u}_\varepsilon 
-(-1)^m\sum_{|\alpha|=|\beta|=m}D_\alpha(\partial_t\tilde A^{\alpha\beta}_{\varepsilon}D^\beta \tilde{\bf u}_\varepsilon)\quad\mbox{in}~Q_{r_1}.
\end{align*}
Applying Lemma \ref{lem loc lq}, using \eqref{est-Dmu} and \eqref{partiat}, we obtain for $p>2$,
\begin{align}\label{est-utDm}
\|\partial_t\tilde{\bf u}_\varepsilon\|_{\mathcal{H}_{p}^m(Q_{{r_1}/2})}&\leq N\big(\|\partial_t\tilde{\bf u}_\varepsilon\|_{L_{2}(Q_{{2r_1/3}})}+\sum_{|\alpha|=m}\|\partial_t\tilde{\bf f}_{\alpha,\varepsilon}\|_{L_{p}(Q_{r_1})}+\| \partial_t^2h_{j_0,\varepsilon}D_d\tilde{\bf u}_\varepsilon\|_{L_{p}(Q_{r_1})}\nonumber\\
&\quad+\|D^m\tilde{\bf u}_\varepsilon\|_{L_{p}(Q_{{2r_1}/3})}\big)\nonumber\\
&\leq N\big(\|\tilde{\bf u}_\varepsilon\|_{L_{p}(Q_{r_1})}+\sum_{|\alpha|=m}\|\tilde{\bf f}_{\alpha,\varepsilon}\|_{L_{p}(Q_{r_1})}+\sum_{|\alpha|=m}\|\partial_t\tilde{\bf f}_{\alpha,\varepsilon}\|_{L_{p}(Q_{r_1})}\big).
\end{align}
Note that the constant in \eqref{est-utDm} may depend on $\varepsilon$. Similarly, for $\alpha=(\alpha_1,\dots,\alpha_{d-1},\alpha_d)=(\alpha',\alpha_d)$, we have
\begin{equation}\label{est-ut00}
D_{y'}^{\alpha'}\tilde{\bf u}_\varepsilon\in \mathcal{H}_{p,loc}^m(Q_{r_1}).
\end{equation}
By the Sobolev embedding theorem for $p>d+2m$, we obtain 
\begin{equation}\label{estDm-1}
\partial_t\tilde{\bf u}_\varepsilon, ~D^{m-1}D_{y'}^{\alpha'}\tilde{\bf u}_\varepsilon\in L_{\infty,loc}(Q_{r_1}).
\end{equation}
From \eqref{eq-tildeu}, we have 
\begin{align*}
D_d^m\big(\sum_{|\beta|=m}\tilde A^{\bar\alpha\beta}_{\varepsilon}D^\beta \tilde{\bf u}_\varepsilon\big)&=(-1)^{m+1}\partial_t\tilde{\bf u}_\varepsilon+(-1)^m\sum_{|\alpha|=m}D_\alpha\tilde{\bf f}_{\alpha,\varepsilon}+(-1)^m \partial_th_{j_0,\varepsilon}D_d\tilde{\bf u}_\varepsilon\\
&\quad-\sum_{\substack{|\alpha|=|\beta|=m\\
\alpha_d<m}}D_d^{\alpha_d}D_{y'}^{\alpha'}(\tilde A^{\alpha\beta}_{\varepsilon}D^\beta \tilde{\bf u}_\varepsilon),
\end{align*}
where $\bar\alpha=me_d$. Then by using \cite[Corollary 4.4]{dk2011}, \eqref{est-utDm}, and \eqref{est-ut00}, we obtain 
$$D_d\big(\sum_{|\beta|=m}\tilde A^{\bar\alpha\beta}_{\varepsilon}D^\beta \tilde{\bf u}_\varepsilon\big)\in L_{p,loc}(Q_{r_1}).$$ 
This together with \eqref{est-utDm}, \eqref{est-ut00} and $\sum_{|\beta|=m}\tilde A^{\bar\alpha\beta}_{\varepsilon}D^\beta \tilde{\bf u}_\varepsilon\in L_{p,loc}(Q_{r_1})$ 
yields 
$$\sum_{|\beta|=m}\tilde A^{\bar\alpha\beta}_{\varepsilon}D^\beta \tilde{\bf u}_\varepsilon\in\mathcal{H}_{p,loc}^1(Q_{r_1}).$$
It follows from the Sobolev embedding theorem for $p>d+2$ that 
$$\sum_{|\beta|=m}\tilde A^{\bar\alpha\beta}_{\varepsilon}D^\beta \tilde{\bf u}_\varepsilon\in L_{\infty,loc}(Q_{r_1}).$$
Using \eqref{ellipticity} and \eqref{estDm-1}, we derive the local boundedness of $D^m\tilde{\bf u}_\varepsilon$ and thus, $D^m{\bf u}_\varepsilon$ is bounded near $z_0$. Then by applying the a priori  estimate in Lemma \ref{Dmubdd} to get a uniform $L_\infty$-estimate of $D^m{\bf u}_\varepsilon$  independent of $\varepsilon$, and finally taking $\varepsilon\rightarrow0^+$, we obtain the boundedness of $D^m{\bf u}$ near $z_0$. Proposition \ref{main prop} is proved.
\end{proof}

\section{Proof of Theorem \ref{mainthm2}}

The estimate \eqref{Dmuholder} follows from Theorem \ref{mainthm} and \cite[Lemma 5.1]{dx2021}. Moreover, combining Lemma \ref{Dmubdd} with the argument in Step 2 of the proof of Proposition \ref{main prop}, we obtain the estimate of $\|D^m{\bf u}\|_{L_\infty}$ in \eqref{edt-ut}. We next show the estimate of $\langle{\bf u}\rangle_{\frac{m+\delta'}{2m}}$ in \eqref{edt-ut}. 

For any $z_{0}=(t_0,x_0)\in  Q_{3/4}$, as mentioned on page \pageref{systemmain}, we assume $z_0\in Q_{3/4}\cap \cQ_{j_{0}}$ for some $j_0\in\{1,\dots,l+1\}$. For $(t,x)\in \cQ_j$ with $\cQ_j$ given in \eqref{def-subdomain}, we can take the piecewise constant functions as follows:
\begin{equation}\label{def-bar-A}
\bar{A}_{z_0}^{\alpha\beta}(t,x)=\bar{A}_{z_0}^{\alpha\beta}(x_d)=\begin{cases}
\lim\limits_{\cQ_j\ni(s,y)\rightarrow(t_0,x'_0,h_j(t_0,x'_0))}A^{\alpha\beta}(s,y),& j>j_0,\\
A^{\alpha\beta}(t_0,x_0),& j=j_0,\\
\lim\limits_{\cQ_j\ni(s,y)\rightarrow(t_0,x'_0,h_{j-1}(t_0,x'_0))}A^{\alpha\beta}(s,y),& j<j_0.
\end{cases}
\end{equation}
Note that $\bar{A}_{z_0}^{\alpha\beta}$ is independent of the radius $r$, which is different from the scenario where the coefficients are of piecewise Dini mean oscillation. With these piecewise constant functions, we denote
\begin{align*}
\mathbb{P}^{z_0}&=\Bigg\{{\bf p}:{\bf p}(x)=\sum_{\substack{0\leq|\beta|\leq m\\
\beta\neq me_d}}\frac{(x_1-x_{0,1})^{\beta_1}\cdots (x_d-x_{0,d})^{\beta_d}}{\beta_1!\cdots\beta_d!}\ell_\beta\\
&\qquad+\int_{x_{0,d}}^{x_{d}}\int_{x_{0,d}}^{s_{m-1}}\dots\int_{x_{0,d}}^{s_{1}}\Big(\bar{A}_{z_0}^{\bar\alpha\bar\alpha}(s)\Big)^{-1}
\Big(\ell_{\bar\alpha}-\sum_{\substack{|\beta|=m\\\beta\neq me_d}}
\bar{A}_{z_0}^{\bar\alpha\beta}(s)\ell_{\beta}\Big)\ dsds_1\dots ds_{m-1} \Bigg\},
\end{align*}
where $\ell_{\beta}$'s are constants and $\bar\alpha=me_d$.

\begin{lemma} \label{lem3.13}
Let $r,q\in (0,\infty)$ and ${\bf p}\in \mathbb{P}^{z_0}$. Suppose that
\begin{equation}\label{est-Dp}
\fint_{B_r(x_0)}|D^\beta{\bf p}(x)|^q\,dx\le C_0^q r^{q(m-|\beta|+\delta')},
\end{equation}
where $C_0, \delta'>0$ are constants and $|\beta|\leq m$. Then we have
$$
|\ell_\beta|\le NC_0 r^{m-|\beta|+\delta'},\ |\beta|\leq m,
$$
where $N>0$ depends only on $d$, $n$, $\nu$, $q$, and $\delta'$.
\end{lemma}

\begin{proof}
From the definition of ${\bf p}$, we have 
\begin{equation*}
D^\beta{\bf p}=\ell_\beta,\quad |\beta|=m,~\beta\neq me_d,
\end{equation*}
which together with \eqref{est-Dp} implies that 
\begin{equation}\label{ell-beata}
|\ell_\beta|\le NC_0 r^{\delta'}, \quad |\beta|=m,~\beta\neq me_d.
\end{equation} 
A direct calculation gives
\begin{equation*}
D_d^m{\bf p}=\Big(\bar{A}_{z_0}^{\bar\alpha\bar\alpha}(x_d)\Big)^{-1}
\Big(\ell_{\bar\alpha}-\sum_{\substack{|\beta|=m\\\beta\neq me_d}}
\bar{A}_{z_0}^{\bar\alpha\beta}(x_d)\ell_{\beta}\Big).
\end{equation*}
Then by using \eqref{ellipticity}, \eqref{est-Dp}, and \eqref{ell-beata}, we obtain
\begin{equation}\label{est-ld}
|\ell_{\bar\alpha}|\le NC_0 r^{\delta'}.
\end{equation}
For $|\beta|=m-1$, further calculations show that, if $\beta\neq (m-1)e_d$, then 
\begin{align*}
D^\beta{\bf p}=\ell_\beta+
{\bf R}_1(x,\ell_{\alpha})1_{|\alpha|=m, \alpha\neq me_d};
\end{align*}
if $\beta=(m-1)e_d$, then
\begin{align*}
D^\beta{\bf p}=\ell_\beta+
{\bf R}_2(x,\ell_{\alpha})1_{|\alpha|=m, \alpha\neq me_d}+\int_{x_{0,d}}^{x_{d}}\Big(\bar{A}_{z_0}^{\bar\alpha\bar\alpha}(s)\Big)^{-1}
\Big(\ell_{\bar\alpha}-\sum_{\substack{|\beta|=m\\\beta\neq me_d}}
\bar{A}_{z_0}^{\bar\alpha\beta}(s)\ell_{\beta}\Big)\ ds,
\end{align*}
where $|{\bf R}_i(x,\ell_{\alpha})1_{|\alpha|=m, \alpha\neq me_d}|\leq NC_0 r^{1+\delta'}$. Combining \eqref{ell-beata} and \eqref{est-ld}, we get 
$$|\ell_\beta|\le NC_0 r^{1+\delta'}, \quad |\beta|=m-1.$$
Similarly, for $0\leq|\beta|\leq m-2$, by induction, we can prove that
$$|\ell_\beta|\le NC_0 r^{m-|\beta|+\delta'}.$$
The proof of the lemma is finished.
\end{proof}

\begin{lemma}
For any $z_{0}\in Q_{1/8}$ and $r\in(0,1/4)$, in the coordinate system associated with $z_0$, we can find
${\bf p}^{r,z_0}:={\bf p}^{r,z_0}(x)$ in the form
\begin{align}\label{def-p}
&\sum_{\substack{0\leq|\beta|\leq m\\
\beta\neq me_d}}\frac{(x_1-x_{0,1})^{\beta_1}\cdots (x_d-x_{0,d})^{\beta_d}}{\beta_1!\cdots\beta_d!}\ell_\beta^{r,z_0}\nonumber\\
&+\int_{x_{0,d}}^{x_{d}}\int_{x_{0,d}}^{s_{m-1}}\dots\int_{x_{0,d}}^{s_{1}}\Big(\bar{A}_{z_0}^{\bar\alpha\bar\alpha}(s)\Big)^{-1}
\Big((-1)^m\bar{\bf f}_{\bar\alpha;z_0}(s)+\ell_{\bar\alpha}^{r,z_0}\nonumber\\
&-\sum_{|\beta|=m,\beta\neq me_d}
\bar{A}_{z_0}^{\bar\alpha\beta}(s)\ell_{\beta}^{r,z_0}\Big)\ dsds_1\dots ds_{m-1}\in {\bf u}_1+\mathbb{P}^{z_0},
\end{align}
where $\ell_\beta^{r,z_0}$ are constants, $\bar{\bf f}_{\bar\alpha;z_0}$ is defined as in \eqref{def-bar-A}, and ${\bf u}_1$ is defined in \eqref{defw u1} with $(\bar{A}_{z_0}^{\bar{\alpha}\bar\alpha}(s))^{-1}\bar{\bf f}_{\bar\alpha;z_0}(s)$ in place of $(\bar{A}_{r,z_0}^{\bar{\alpha}\bar\alpha}(s))^{-1}\bar{\bf f}_{\bar\alpha;r,z_0}(s)$,
such that
\begin{align}\label{difference Du p0}
\fint_{Q_{r}(z_{0})}|D^{\alpha}({\bf u}-{\bf p}^{r,z_{0}})|^{1/2}\leq NC_{0}^{1/2}r^{(m-|\alpha|+\delta')/2},
\end{align}
where $\delta'=\min\{\delta,\frac{\mu}{1+\mu},2m\gamma_0-1\}$, $|\alpha|\leq m$, and
\begin{equation}\label{def C0}
C_{0}=\sum_{j=1}^{M}|{\bf f}_\alpha|_{\delta/2,\delta;\overline{\cQ}_{j}}+\|{\bf u}\|_{L_{p}(\cQ)}+\|D{\bf u}\|_{L_{1}(\cQ)}.
\end{equation}
\end{lemma}

\begin{proof}
(1) For simplicity, we assume that $x_{0}=0$. Let ${\bf v}$ be the solution of  \eqref{sol-v}. Then similar to \eqref{holder v bar}, using \eqref{Dmv} with $\beta=0$, \eqref{bfF}, the same argument that led to \cite[(3.7)]{dx2019}, and Lemma \ref{Dmubdd}, we obtain
\begin{align}\label{holder v 00}
\left(\fint_{Q_{r/2}(z_{0})}|{\bf v}|^{1/2}\ dxdt\right)^{2}\leq Nr^{m+\delta'}\Big(\|D^m{\bf u}\|_{L_{\infty}(Q_{r}(z_{0}))}+\sum_{j=1}^{M}|{\bf f}_\alpha|_{\delta/2,\delta;\overline{\cQ}_{j}}\Big)\leq NC_{0}r^{m+\delta'}.
\end{align}
Define
\begin{align*}
{\bf p}_1(x)&={\bf w}(z_0)+\sum_{\substack{1\leq|\beta|\leq m\\\beta\neq me_d}}\frac{x_1^{\beta_1}\dots x_d^{\beta_d}}{\beta_1!\dots\beta_d!}D^\beta{\bf w}(z_{0})+\int_{0}^{x_{d}}\int_{0}^{s_{m-1}}\dots\int_{0}^{s_{1}}\Big(\bar{A}_{z_0}^{\bar\alpha\bar\alpha}(s)\Big)^{-1}\\
&\qquad \cdot\Big({\bf W}(z_{0})-\sum_{\substack{|\beta|=m\\\beta\neq me_d}}\bar{A}_{z_0}^{\bar\alpha\beta}(s)D^{\beta}{\bf w}(z_{0})\Big)\ dsds_1\dots ds_{m-1}\in \mathbb{P}^{z_0},
\end{align*}
where ${\bf w}$ and ${\bf W}$ are defined in \eqref{def_w} and \eqref{def_V-W}, respectively. 
Then a direct calculation yields 
$$\widetilde D^\alpha{\bf p}_1=\widetilde D^\alpha{\bf w}(z_0),\quad \sum_{|\beta|=m}\bar{A}_{z_0}^{\bar\alpha\beta}(x_d)D^\beta{\bf p}_1={\bf W}(z_0),$$
where $\widetilde{D}^\alpha{\bf w}$ is the collection of $D^\alpha{\bf w}$, $|\alpha|=m$ and $\alpha\neq me_d$. Since ${\bf w}$ satisfies \eqref{eq_w}, applying Lemma \ref{lemma xn} and the local estimate in Lemma \ref{lem loc lq}  with a suitable scaling, we have 
\begin{equation}\label{est-wpt-t}
\|{\bf w}_t\|_{L_\infty(Q_{\kappa r}(z_{0}))}
\leq N r^{-2m-2(d+2m)}\|{\bf w}\|_{L_{1/2}(Q_{r/2}(z_{0}))},
\end{equation}
$$
\|\widetilde D^\alpha({\bf w}-{\bf p}_1)\|_{L_\infty(Q_{\kappa r}(z_{0}))}
\leq N\kappa r^{-m-2(d+2m)}\|{\bf w}\|_{L_{1/2}(Q_{r/2}(z_{0}))},
$$
and
\begin{align*}
\big\|\sum_{|\beta|=m}\bar{A}_{z_0}^{\bar\alpha\beta}(x_d)D^\beta({\bf w}-{\bf p}_1)\big\|_{L_\infty(Q_{\kappa r}(z_{0}))}
\leq N\kappa r^{-m-2(d+2m)}\|{\bf w}\|_{L_{1/2}(Q_{r/2}(z_{0}))},
\end{align*}
where $\kappa\in(0,1/2)$. Thus, together with \eqref{ellipticity}, we derive 
\begin{equation}\label{estDmw-p1}
\|D^m({\bf w}-{\bf p}_1)\|_{L_\infty(Q_{\kappa r}(z_{0}))}
\le N\kappa r^{-m-2(d+2m)}\|{\bf w}\|_{L_{1/2}(Q_{r/2}(z_{0}))}.
\end{equation}
Note that ${\bf w}-{\bf p}$ satisfies the same equation as ${\bf w}$ for any ${\bf p}\in \mathbb{P}^{z_0}$, one can infer from \eqref{est-wpt-t} and \eqref{estDmw-p1} that 
\begin{equation}\label{est-wpt-t11}
\|{\bf w}_t\|_{L_\infty(Q_{\kappa r}(z_{0}))}
\leq N r^{-2m-2(d+2m)}\|{\bf w}-{\bf p}\|_{L_{1/2}(Q_{r/2}(z_{0}))},
\end{equation}
and
\begin{equation}\label{est-Dmw-pp1}
\|D^m({\bf w}-{\bf p}_1)\|_{L_\infty(Q_{\kappa r}(z_{0}))}\le N\kappa r^{-m-2(d+2m)}\|{\bf w}-{\bf p}\|_{L_{1/2}(Q_{r/2}(z_{0}))}.
\end{equation}
Since $D^\beta({\bf w}-{\bf p}_1)(z_0)=0$, $|\beta|\leq m-1$, by using \eqref{est-wpt-t11} and \eqref{est-Dmw-pp1}, we have
\begin{align}\label{est-w-p1}
&\fint_{Q_{\kappa r}(z_{0})}|{\bf w}-{\bf p}_1|^{1/2}
\le N\fint_{Q_{\kappa r}(z_{0})}(|{\bf w}-{\bf w}(t_0,x)|+|{\bf w}(t_0,x)-{\bf p}_1|^{1/2}\nonumber\\
&\leq N(\kappa r)^{m}\|{\bf w}_t\|^{1/2}_{L_\infty(Q_{\kappa r}(z_{0}))}
+N(\kappa r)^{m/2}\|D^m({\bf w}-{\bf p}_1)\|_{L_\infty(Q_{\kappa r}(z_{0}))}^{1/2}\leq N\kappa^{\frac{m+1}{2}}\fint_{Q_{r/2}(z_{0})}|{\bf w}-{\bf p}|^{1/2}.
\end{align}

Recalling ${\bf w}={\bf u}-{\bf u}_1-{\bf v}$ in \eqref{def_w}, using the triangle inequality,  \eqref{holder v 00}, and \eqref{est-w-p1}, we have for any ${\bf p}\in\mathbb{P}^{z_0}$,
\begin{align}\label{iterau-p}
&\fint_{Q_{\kappa r}(z_{0})}|{\bf u}-{\bf u}_1-{\bf p}_1|^{1/2}\ dz\nonumber\\
&\leq\fint_{Q_{\kappa r}(z_{0})}|{\bf w}-{\bf p}_1|^{1/2}\ dz+\fint_{Q_{\kappa r}(z_{0})}|{\bf v}|^{1/2}\ dz\nonumber\\
&\leq N\kappa^{\frac{m+1}{2}}\fint_{Q_{r/2}(z_{0})}|{\bf w}-{\bf p}|^{1/2}\ dz+NC_{0}^{1/2}\kappa^{-(d+2m)}r^{\frac{1}{2}(m+\delta')}
\nonumber\\
&\leq N\kappa^{\frac{m+1}{2}}\fint_{Q_{r/2}(z_{0})}|{\bf u}-{\bf u}_1-{\bf p}|^{1/2}\ dz+NC_{0}^{1/2}\kappa^{-(d+2m)}r^{\frac{1}{2}(m+\delta')},
\end{align}
where $C_{0}$ is defined in \eqref{def C0}. Denote
$$
F(r):=\inf_{{\bf p}\in\mathbb{P}^{z_0}}
\fint_{Q_{r}(z_{0})}|{\bf u}-{\bf u}_1-{\bf p}|^{1/2}\ dz.
$$
Then it follows from \eqref{iterau-p} that, for any  $r\in (0,1/4)$ and $\kappa\in(0,1/2)$,
\begin{align*}
F(\kappa r)\leq N\kappa^{\frac{m+1}{2}} F(r)+NC_{0}^{1/2}\kappa^{-(d+2m)}r^{\frac{1}{2}(m+\delta')}.
\end{align*}	
Using a well-known iteration argument (see, for instance, \cite[ Chapter V, Lemma 3.1]{g1983}), we have
\begin{align*}
F(r)\leq NC_{0}^{1/2}r^{\frac{1}{2}(m+\delta')}.
\end{align*}
Thus, there exists ${\bf p}^{r,z_0}$ in the form of \eqref{def-p} such that \eqref{difference Du p0} holds with $\alpha=0$. 

(2) By ${\bf p}^{r,z_{0}}$ as given in the above and using direct computations, we obtain
\begin{align}\label{Dmpr}
\sum_{|\beta|=m}\bar{A}_{z_0}^{\bar\alpha\beta}(x_d)D^{\beta}{\bf p}^{r,z_{0}}=
(-1)^m\bar{\bf f}_{\bar\alpha;z_0}(x_d)+\ell_{\bar\alpha}^{r,z_0},
\end{align}
and thus,
\begin{align*}
&({\bf u}-{\bf p}^{r,z_{0}})_t+(-1)^{m}\sum_{|\alpha|=|\beta|=m}D^{\alpha}(\bar{A}_{z_0}^{\alpha\beta}(x_d)D^{\beta}({\bf u}-{\bf p}^{r,z_{0}}))\\
&=(-1)^m\sum_{|\alpha|=|\beta|=m}D^\alpha\big((\bar{A}_{z_0}^{\alpha\beta}(x_d)-A^{\alpha\beta}(t,x))D^{\beta}{\bf u}\big)+\sum_{|\alpha|=m}D^\alpha({\bf f}_\alpha(t,x)-\bar{{\bf f}}_{\alpha;z_0}(x_d)).
\end{align*}
Denote $\tilde{\bf w}:={\bf w}-{\bf p}_1$, where ${\bf w}$ is defined in \eqref{def_w}. Then $\tilde{\bf w}$ satisfies
\begin{align*}
\tilde{\bf w}_t+(-1)^{m}\sum_{|\alpha|=|\beta|=m}D^{\alpha}(\bar{A}_{z_0}^{\alpha\beta}(x_d)D^{\beta}\tilde{\bf w})=0\quad\mbox{in}~Q_{r/2}(z_0).
\end{align*}
By using Lemma \ref{lem loc lq} with $p=2$, the triangle inequality, \eqref{difference Du p0} with $\alpha=0$, and \eqref{holder v 00}, we have for $1\leq|\alpha|\leq m$,
\begin{align*}
\|D^{\alpha}\tilde{\bf w}\|_{L_{2}(Q_{r/3}(z_{0}))}
&\leq Nr^{-|\alpha|+\frac{d+2m}{2}-2(d+2m)}\big(\|{\bf u}-{\bf p}^{r,z_{0}}\|_{L_{1/2}(Q_{r/2}(z_{0}))}+\|{\bf v}\|_{L_{1/2}(Q_{r/2}(z_{0}))}\big)\\
&\leq NC_{0}r^{m-|\alpha|+\delta'+\frac{d+2m}{2}}.
\end{align*}
Together with H\"{o}lder's inequality, we obtain
\begin{align}\label{holder w tilde}
\fint_{Q_{r/3}(z_{0})}|D^{\alpha}\tilde{\bf w}|^{1/2}\leq\left(\fint_{Q_{r/3}(z_{0})}|D^{\alpha}\tilde{\bf w}|^{2}\right)^{\frac{1}{4}}\leq NC_{0}^{1/2}r^{(m-|\alpha|+\delta')/2}.
\end{align}
Moreover, similar to \eqref{holder v 00}, we have 
\begin{align*}
\left(\fint_{Q_{r/2}(z_{0})}|D^{\alpha}{\bf v}|^{1/2}\ dxdt\right)^{2}\leq NC_{0}r^{m-|\alpha|+\delta'},
\end{align*}
where $1\leq|\alpha|\leq m$. 
This, in combination with the triangle inequality and \eqref{holder w tilde}, yields \eqref{difference Du p0} with $1\leq|\alpha|\leq m$.
The lemma is proved.
\end{proof}

\begin{lemma}
For any $z_0\in Q_{1/8}$ and $r\in(0,1/4)$, we have 
\begin{equation}\label{limit lbeta}
D^{\beta}{\bf u}(z_0)=\ell_\beta^{z_0},\quad |D^{\beta}{\bf u}(z_0)-\ell_\beta^{r,z_0}|\le NC_0 r^{m-|\beta|+\delta'},\quad|\beta|\leq m,~|\beta|\neq me_d,
\end{equation}
and
\begin{align}\label{limit ld}
{\bf U}(z_0)=\ell_{\bar\alpha}^{z_0},\quad |{\bf U}(z_0)-\ell_{\bar\alpha}^{r,z_0}|\le NC_0 r^{\delta'},\quad\bar\alpha=me_d,
\end{align}
where $\ell_\beta^{r,z_0}$ are defined in \eqref{def-p}, $\ell_\beta^{z_0}$ are the limits of $\ell_\beta^{r,z_0}$ as $r\rightarrow0$, $|\beta|\leq m$, ${\bf U}$ is defined in \eqref{defU} with $|\beta|=m$, the constant $N$ depends on $d$, $n$, $m$, $\nu$, and $\delta'$.
\end{lemma}

\begin{proof}
For simplicity, we assume $x_0=0$. It follows from the triangle inequality and \eqref{difference Du p0} that
\begin{align*}
\fint_{Q_{r}(z_{0})}|D^\beta({\bf p}^{r,z_0}-{\bf p}^{2r,z_{0}})|^{1/2}
&\le \fint_{Q_{r}(z_{0})}|D^\beta({\bf u}-{\bf p}^{r,z_0})|^{1/2}+|D^\beta({\bf u}-{\bf p}^{2r,z_0})|^{1/2}\\
&\leq NC_{0}^{1/2}r^{(m-|\beta|+\delta')/2},
\end{align*}
where $|\beta|\leq m$.  From \eqref{def-p}, we have 
\begin{align*}
{\bf p}^{r,z_{0}}-{\bf p}^{2r,z_{0}}
&=\ell_0^{r,z_0}-\ell_0^{2r,z_0}+(-1)^m\sum_{\substack{1\leq|\beta|\leq m\\
\beta\neq me_d}}\frac{x_1^{\beta_1}\cdots x_d^{\beta_d}}{\beta_1!\cdots\beta_d!}(\ell_\beta^{r,z_0}-\ell_\beta^{2r,z_0})\nonumber\\
&+(-1)^m\int_{0}^{x_{d}}\int_{0}^{s_{m-1}}\dots\int_{0}^{s_{1}}\Big(\bar{A}_{z_0}^{\bar\alpha\bar\alpha}(s)\Big)^{-1}
\Big(\ell_{\bar\alpha}^{r,z_0}-\ell_{\bar\alpha}^{2r,z_0}\nonumber\\
&-\sum_{|\beta|=m,\beta\neq me_d}
\bar{A}_{z_0}^{\bar\alpha\beta}(s)(\ell_{\beta}^{r,z_0}-\ell_{\beta}^{2r,z_0})\Big)\ dsds_1\dots ds_{m-1}\in\mathbb{P}^{z_0}.
\end{align*}
Applying Lemma \ref{lem3.13}, we have 
\begin{equation}\label{difference ell}
|\ell^{r,z_0}_\beta-\ell^{2r,z_0}_\beta|\le NC_0 r^{m-|\beta|+\delta'},\ |\beta|\leq m.
\end{equation}
This implies that the limits of $\ell^{r,z_0}_\beta$ exist and are denoted by $\ell_\beta^{z_0}$ for $|\beta|\leq m$.
Combining \eqref{difference Du p0} and the continuity of $D^\beta{\bf u}$, $|\beta|\leq m-1$, one can get \eqref{limit lbeta}  for $|\beta|\leq m-1$.

Next we prove \eqref{limit lbeta} for $|\beta|=m$ and \eqref{limit ld}. From \eqref{difference Du p0}, one has that
\begin{align}\label{difference Du p}
\fint_{Q_{r/2}(z_{0})}|\sum_{|\beta|=m}\bar{A}^{\bar\alpha\beta}D^{\beta}({\bf u}-{\bf p}^{r,z_{0}})|^{1/2}\leq NC_{0}^{1/2}r^{\delta'/2}.
\end{align}
In view of the definition of ${\bf U}$ in \eqref{defU} with $|\beta|=m$, \eqref{Dmpr}, the triangle inequality, \eqref{difference Du p}, and \cite[Lemma 5.1]{dx2021}, we have 
\begin{align}\label{difference Du U}
&\fint_{Q_{r/3}(z_{0})}|{\bf U}-\ell_{\bar\alpha}^{r,z_{0}}|^{1/2}\nonumber\\
&\leq \fint_{Q_{r/3}(z_{0})}\big|\sum_{|\beta|=m}(A^{\bar\alpha\beta}D^{\beta}{\bf u}-\bar{A}_{z_0}^{\bar\alpha\beta}(x_d)D^{\beta}{\bf p}^{r,z_{0}})+(-1)^m{\bf f}_{\bar\alpha}-(-1)^m\bar{\bf f}_{\bar\alpha;z_0}(x_d)\big|^{1/2}\nonumber\\
&\leq NC_{0}^{1/2}r^{\delta'/2}.
\end{align}
Therefore, the desired estimates follow from \eqref{difference ell}, \eqref{difference Du U},  and the continuity of $\widetilde D^{\beta}{\bf u}$ and ${\bf U}$ in \eqref{Dmuholder}.
\end{proof}

Now we are ready to complete the proof of the estimate of $\langle{\bf u}\rangle_{\frac{m+\delta'}{2m}}$ in \eqref{edt-ut}.

\begin{proof}[Proof of the estimate of $\langle{\bf u}\rangle_{\frac{m+\delta'}{2m}}$ in \eqref{edt-ut}]
For any point $z_0=(t_0,x_0), z_1=(t_1,x_0)\in Q_{3/4}$, denote $r:=|t_0-t_1|^{\frac{1}{2m}}$. If $r>1/32$, then by \eqref{estholderu}, we have 
\begin{equation*}
|{\bf u}(z_{0})-{\bf u}(z_{1})|\leq 2\|{\bf u}\|_{L_\infty(Q_{3/4})}\leq N |t_0-t_1|^{\frac{m+\delta'}{2m}}\big(\|{\bf u}\|_{L_{p}(Q_{1})}
+\sum_{|\alpha|\leq m}\|{\bf f}_\alpha\|_{L_{\infty}(Q_{1})}\big).
\end{equation*}
Next we consider the case of $r\leq 1/32$. Without loss of generality, we may assume $x_0=0$ and $z_1=(t_0-r^{2m},0)$ are in the same subdomain. 
Using the triangle inequality, we have 
\begin{equation}\label{estimate u}
|{\bf u}(z_{0})-{\bf u}(z_{1})|^{1/2}\leq|{\bf u}(z_{0})-\ell_0^{2r,z_0}|^{1/2}+|\ell_0^{2r,z_0}-\ell_0^{2r,z_1}|^{1/2}+|{\bf u}(z_{1})-\ell_0^{2r,z_1}|^{1/2}.
\end{equation}
We will use $x$ and $y$ to denote the same point under the coordinate system associated with $z_0$ and $z_1$, respectively. Moreover, it follows from \cite[Lemma 3.3]{dx2021} with $\omega_1(r)$ replaced by $r^{\delta'}$ that
\begin{equation}\label{difference x tilde x}
|x-y|\leq Nr^{1+\delta'}.
\end{equation}
To estimate $|\ell_0^{2r,z_0}-\ell_0^{2r,z_1}|$, we first use the definition of ${\bf p}^{r,z_0}$ in \eqref{def-p} with  $x_0=0$ to obtain
\begin{align}\label{uz0-z1}
&|\ell_0^{2r,z_0}-\ell_0^{2r,z_1}|^{1/2}\nonumber\\
&\leq|{\bf p}^{2r,z_0}-{\bf p}^{2r,z_1}|^{1/2}+\Big|\sum_{\substack{1\leq|\beta|\leq m\\
\beta\neq me_d}}\frac{1}{\beta_1!\cdots\beta_d!}(\ell_\beta^{2r,z_0}x_1^{\beta_1}\cdots x_d^{\beta_d}-\ell_\beta^{2r,z_1}y_1^{\beta_1}\cdots y_d^{\beta_d})\Big|^{1/2}\nonumber\\
&\quad+\Bigg|\int_{0}^{x_{d}}\int_{0}^{s_{m-1}}\dots\int_{0}^{s_{1}}\Big(\bar{A}_{z_{0}}^{\bar\alpha\bar\alpha}(s)\Big)^{-1}
\Big((-1)^m\bar{\bf f}_{\bar\alpha;z_{0}}(s)+\ell_{\bar\alpha}^{2r,z_0}\nonumber\\
&\qquad-\sum_{|\beta|=m,\beta\neq me_d}
\bar{A}_{z_{0}}^{\bar\alpha\beta}(s)\ell_{\beta}^{2r,z_0}\Big)\ dsds_1\dots ds_{m-1}\nonumber\\
&\qquad-\int_{0}^{y_{d}}\int_{0}^{s_{m-1}}\dots\int_{0}^{ s_{1}}\Big(\bar{A}_{z_{1}}^{\bar\alpha\bar\alpha}(s)\Big)^{-1}
\Big((-1)^m\bar{\bf f}_{\bar\alpha;z_{1}}(s)+\ell_{\bar\alpha}^{2r,z_1}\nonumber\\
&\qquad-\sum_{|\beta|=m,\beta\neq me_d}
\bar{A}_{z_{1}}^{\bar\alpha\beta}(s)\ell_{\beta}^{2r,z_1}\Big)\ dsds_1\dots d s_{m-1}\Bigg|^{1/2},
\end{align}
where we used  $g_{_{z_0}}$ and $g_{_{z_1}}$ to denote the same function $g$ in the coordinate system associated with $z_0$ and $z_1$, respectively. We are going to estimate the three terms on the right-hand side of \eqref{uz0-z1} one by one. 

In view of the triangle inequality and \eqref{difference Du p0}, we have $Q_{r}(z_{1})\subset Q_{2r}(z_{0})$ and 
\begin{align}\label{est first}
\fint_{Q_{r}(z_{1})}|{\bf p}^{2r,z_0}-{\bf p}^{2r,z_1}|^{1/2}&\leq\fint_{Q_{r}(z_{1})}|{\bf u}(z)-{\bf p}^{2r,z_0}|^{1/2}+\fint_{Q_{r}(z_{1})}|{\bf u}(z)-{\bf p}^{2r,z_1}|^{1/2}\nonumber\\
&\leq\fint_{Q_{2r}(z_{0})}|{\bf u}(z)-{\bf p}^{2r,z_0}|^{1/2}+NC_{0}^{1/2}r^{(m+\delta')/2}\leq NC_{0}^{1/2}r^{(m+\delta')/2}.
\end{align}
Under the coordinate system associated with $z_{0}$, by using the triangle inequality, Lemma \ref{lem3.13}, \eqref{difference x tilde x}, \eqref{limit lbeta}, and \eqref{Dmuholder}, we have for $1\leq|\beta|\leq m$ and $\beta\neq me_d$,
\begin{align}\label{est second}
&|\ell_\beta^{2r,z_0}x_1^{\beta_1}\cdots x_d^{\beta_d}-\ell_\beta^{2r,z_1}y_1^{\beta_1}\cdots y_d^{\beta_d}|\nonumber\\
&\leq|(\ell_\beta^{2r,z_{0}}-\ell_\beta^{2r,z_{1}})x_1^{\beta_1}\cdots x_d^{\beta_d}+\ell_\beta^{2r,z_{1}} (x_1^{\beta_1}\cdots x_d^{\beta_d}-y_1^{\beta_1}\cdots y_d^{\beta_d})|\nonumber\\
&\leq r^{|\beta|}\Big(|\ell_\beta^{2r,z_{0}}-D^{\beta}{\bf u}(z_{0})|+|D^{\beta}{\bf u}(z_{0})-D^{\beta}{\bf u}(z_{1})|+|D^{\beta}{\bf u}(z_{1})-D^{y_\beta}{\bf u}(z_{1})|\nonumber\\
&\quad+|D^{y_{\beta}}{\bf u}(z_{1})-\ell_\beta^{2r,z_{1}}|\Big)+NC_{0}r^{m+\delta'}\leq NC_{0}r^{m+\delta'},
\end{align}
where $D^{y_\beta}$ denotes the derivatives under the coordinate system associated with $z_1$, \eqref{limit lbeta} and \eqref{limit ld} hold true with $z_0$ replaced by $z_1$. For the third term in \eqref{uz0-z1},  note that 
\begin{align}\label{est third1}
&\Big|\int_{0}^{x_{d}}\int_{0}^{s_{m-1}}\dots\int_{0}^{s_{1}}\big(\bar{A}_{z_{0}}^{\bar\alpha\bar\alpha}(s)\big)^{-1}\bar{\bf f}_{\bar\alpha;z_{0}}(s)-\int_{0}^{y_{d}}\int_{0}^{ s_{m-1}}\dots\int_{0}^{s_{1}}\big(\bar{A}_{z_{1}}^{\bar\alpha\bar\alpha}(s)\big)^{-1}
\bar{\bf f}_{\bar\alpha;z_{1}}(s)\Big|\nonumber\\
&\leq\Big|\int_{0}^{x_{d}}\int_{0}^{s_{m-1}}\dots\int_{0}^{s_{1}}\big(\bar{A}_{z_{0}}^{\bar\alpha\bar\alpha}(s)\big)^{-1}(\bar{\bf f}_{\bar\alpha;z_{0}}(s)-\bar{\bf f}_{\bar\alpha;z_{1}}(s))\Big|\nonumber\\
&\quad+\Big|\int_{0}^{x_{d}}\int_{0}^{s_{m-1}}\dots\int_{0}^{s_{1}}\big((\bar{A}_{z_{0}}^{\bar\alpha\bar\alpha}(s))^{-1}-(\bar{A}_{z_{1}}^{\bar\alpha\bar\alpha}(s))^{-1}\big)\bar{\bf f}_{\bar\alpha;z_{1}}(s)\Big|\nonumber\\
&\quad+\Big|\int_{y_{d}}^{x_{d}}\int_{0}^{s_{m-1}}\dots\int_{0}^{s_{1}}\big|(\bar{A}_{z_{1}}^{\bar\alpha\bar\alpha}(s))^{-1}
\bar{\bf f}_{\bar\alpha;z_{1}}(s)\big|\ dsds_1\dots ds_{m-1}\Big|.
\end{align}	
By using the triangle inequality, \eqref{ellipticity}, \eqref{difference x tilde x}, the assumptions of $\bar{\bf f}_{\bar\alpha}$ and  $h_j$ defined in Page \pageref{systemmain}, and $z_0$ and $z_1$ are in the same subdomain,  we have
\begin{align*}
\Big|\int_{0}^{x_{d}}\int_{0}^{s_{m-1}}\dots\int_{0}^{s_{1}}\big(\bar{A}_{z_{0}}^{\bar\alpha\bar\alpha}(s)\big)^{-1}(\bar{\bf f}_{\bar\alpha;z_{0}}(s)-\bar{\bf f}_{\bar\alpha;z_{1}}(s))\Big|\leq NC_{0}r^{m+\delta'},
\end{align*}
which is also true for the last two terms on the right-hand side of \eqref{est third1}.
Similarly,  by the triangle inequality, \eqref{limit lbeta}, and \eqref{limit ld}, we can estimate the rest of terms, and we conclude that the third term on the right-hand side of \eqref{uz0-z1} is bounded by $NC_{0}^{1/2}r^{(m+\delta')/2}$. Then by taking the average over $z\in Q_{r}(z_{1})$ and taking the $1/2$-th root in \eqref{estimate u}, combining \eqref{est first} and \eqref{est second},  we have
\begin{equation*}%\label{esi u z}
|{\bf u}(z_0)-{\bf u}(z_1)|\leq NC_0r^{m+\delta'}.
\end{equation*}
Therefore, we finish the proof of \eqref{edt-ut}.
\end{proof}

\appendix

\section{}\label{Append}
In this section, we shall show that we only need to consider the parabolic systems with leading order terms by moving all the lower-order terms to the right-hand side of \eqref{systems}. To this end, let us first present the following two auxiliary results in Lemmas \ref{lemlocal} and \ref{lem loc lq}.

From Lemma \ref{solvability}, by a similar localization argument that led to \cite[Lemma 4]{d2012}, we obtain the following local estimate.
\begin{lemma}\label{lemlocal}
Let $p\in(1,\infty)$ and ${\bf f}_{\alpha}\in L_{p}(Q_{1})$. Assume $A^{\alpha\beta}$ satisfies \eqref{BMO} with a sufficiently small constant $\gamma_{0}=\gamma_{0}(d,n,m,p,\nu)\in (0,1/2)$ and ${\bf u}\in \mathcal{H}_{p,\text{loc}}^{m}$ satisfies \eqref{systems} in $Q_{1}$. Then there exists a constant $N=N(n,d,m,p,\nu,\Lambda,r_0)$ such that 
$$\|{\bf u}\|_{\mathcal{H}_{p}^{m}(Q_{1/2})}\leq N\big(\|{\bf u}\|_{L_{p}(Q_{1})}+\sum_{|\alpha|\leq m}\|{\bf f}_{\alpha}\|_{L_{p}(Q_{1})}\big).$$
\end{lemma}

By Sobolev embedding theorem, Lemma \ref{lemlocal}, the interpolation inequality,  iteration arguments (see, for instance, \cite[pp. 81--82]{g1993}), and a bootstrap argument, we have the following result. 
\begin{lemma}\label{lem loc lq}
Let $0<p<1<q<\infty$. Assume $A^{\alpha\beta}$ satisfies \eqref{BMO} with a sufficiently small constant $\gamma_{0}=\gamma_{0}(d,n,m,p,q,\Lambda,\nu)\in (0,1/2)$ and ${\bf u}\in C_{\text{loc}}^{\infty}(\mathbb R^{d+1})$ satisfies \eqref{systems} in $Q_{1}$, where ${\bf f}_{\alpha}\in L_{q}(Q_{1})$. Then there exists a constant $N=N(n,d,m,p,q,\nu,\Lambda,r_0)$ such that 
$$\|{\bf u}\|_{\mathcal{H}_{q}^{m}(Q_{1/2})}\leq N\big(\|{\bf u}\|_{L_{p}(Q_{1})}+\sum_{|\alpha|\leq m}\|{\bf f}_{\alpha}\|_{L_{q}(Q_{1})}\big).$$
\end{lemma}

Next we shall use Lemmas \ref{lemlocal} and \ref{lem loc lq} to conclude that we only need to consider the parabolic systems without lower-order terms. By Lemma \ref{lem loc lq}, we have for some $q$ satisfying $mq>d+2m$,
\begin{equation}\label{est u H1q}
\|{\bf u}\|_{\mathcal{H}_{q}^{m}(Q_{1/2})}\leq N\big(\|{\bf u}\|_{L_{p}(Q_{1})}
+\sum_{|\alpha|\leq m}\|{\bf f}_\alpha\|_{L_{\infty}(Q_{1})}\big).
\end{equation}
By the Sobolev embedding theorem for $mq>d+2m$, we obtain
\begin{equation}\label{estholderu}
\|{\bf u}\|_{C^{\frac{1}{2}(1-\frac{d+2m}{mq}),m-\frac{d+2m}{q}}(Q_{1/2})}\leq N\big(\|{\bf u}\|_{L_{p}(Q_{1})}
+\sum_{|\alpha|\leq m}\|{\bf f}_\alpha\|_{L_{\infty}(Q_{1})}\big).
\end{equation}
Now we rewrite \eqref{systems} as
\begin{align*}
{\bf u}_{t}+(-1)^{m}\sum_{|\alpha|=|\beta|= m}D^{\alpha}(A^{\alpha\beta}D^{\beta}{\bf u})&=\sum_{|\alpha|\leq m}D^{\alpha}{\bf f}_{\alpha}+(-1)^{m+1}\sum_{|\alpha|=m, |\beta|\leq m-1}D^{\alpha}(A^{\alpha\beta}D^{\beta}{\bf u})\\
&\quad+(-1)^{m+1}\sum_{|\alpha|\leq m-1, |\beta|\leq m}D^{\alpha}(A^{\alpha\beta}D^{\beta}{\bf u})\quad \mbox{in}~Q_{1}.
\end{align*}
Denote 
\begin{equation*}
{\bf F}_\alpha:={\bf f}_{\alpha}+(-1)^{m+1}\sum_{|\beta|\leq m}A^{\alpha\beta}D^{\beta}{\bf u}.
\end{equation*}
For each $\gamma$ with $|\gamma|\leq m-1$, let ${\bf v}\in W_{q}^{1,2m}(Q_{1})$ satisfy
\begin{align*}
\begin{cases}
{\bf v}_{t}+(-1)^{m}\sum_{|\alpha|=|\beta|= m}\delta^{\alpha\beta}\mathcal{I}_{n\times n}D^{\alpha}D^{\beta}{\bf v}={\bf F}_\gamma\chi_{Q_{1/2}}
&\ \mbox{in}~Q_{1},\\
{\bf v}=|D{\bf v}|=\dots=|D^{m-1}{\bf v}|=0&\ \mbox{on}~\partial_p Q_1.
\end{cases}
\end{align*}
Then by the global $W_{q}^{1,2m}$-estimate (see, for instance, \cite[Theorem 6]{dk2011ARMA} by considering $e^{-(\lambda_0+1)t}{\bf u}$ instead of ${\bf u}$ there), we have
\begin{align*}%\label{W2p v}
\|{\bf v}\|_{W_{q}^{1,2m}(Q_{1})}\leq N\Big(\sum_{|\alpha|\leq m}\|D^\alpha{\bf u}\|_{L_{q}(Q_{1/2})}+\sum_{|\alpha|\leq m-1}\|{\bf f}_\alpha\|_{L_{q}(Q_{1/2})}\Big).
\end{align*}
Using \eqref{est u H1q} with $2mq>d+2m$, we get ${\bf v}\in C^{1-\frac{d+2m}{2mq},2m-\frac{d+2m}{q}}(Q_{1/2})$ (see \cite{i1962,i196200} and \cite[Theorems 2.5, 2.7]{s1967}), with
\begin{equation}\label{est Dv Holder}
\|{\bf v}\|_{C^{1-\frac{d+2m}{2mq},2m-\frac{d+2m}{q}}(Q_{1/2})}\leq N\big(\|{\bf u}\|_{L_{p}(Q_{1})}
+\sum_{|\alpha|\leq m}\|{\bf f}_\alpha\|_{L_{\infty}(Q_{1})}\big).
\end{equation}
Denote ${\bf w}:={\bf u}-\sum_{|\gamma|\leq m-1}D^\gamma{\bf v}$, then ${\bf w}$ satisfies
\begin{align*}
{\bf w}_{t}+(-1)^{m}\sum_{|\alpha|=|\beta|= m}D^{\alpha}(A^{\alpha\beta}D^{\beta}{\bf w})=\sum_{|\alpha|=m}D^{\alpha}{\bf g}_{\alpha}\quad \mbox{in}~Q_{1/2},
\end{align*}
where
\begin{align*}
{\bf g}_{\alpha}={\bf f}_{\alpha}+(-1)^{m+1}\sum_{|\beta|\leq m-1}A^{\alpha\beta}D^{\beta}{\bf u}+(-1)^{m+1}\sum_{|\beta|=m}(A^{\alpha\beta}-\delta^{\alpha\beta}\mathcal{I}_{n\times n})D^{\beta}\big(\sum_{|\gamma|\leq m-1}D^\gamma{\bf v}\big).
\end{align*}
In view of the assumption that ${\bf f}_\alpha\in L_\infty(Q_1)$, the boundedness of $A^{\alpha\beta}$, \eqref{estholderu}, and \eqref{est Dv Holder}, we have 
$$\|{\bf g}_{\alpha}\|_{L_{\infty}(Q_{1/2})}\leq N\left(\|{\bf u}\|_{L_{p}(Q_{1})}+\|{\bf f}_{\alpha}\|_{L_{\infty}(Q_{1})}\right).$$
Moreover, by the triangle inequality, using \eqref{estholderu} and \eqref{est Dv Holder} again, one can see that ${\bf g}_{\alpha}$ is piecewise DMO. 
Therefore, it suffices to consider the parabolic systems without lower-order terms.

%%%%%%%%%%%%%%%%%%%%%%%%%%%%%%%%%%%%%%%

\end{document}